\documentclass[cmp]{svjour}
\usepackage{graphics, color, amsmath, amsfonts, amssymb, amscd}
\usepackage{graphicx}
\usepackage[colorlinks]{hyperref}
\DeclareGraphicsExtensions{.ps,.eps}
\usepackage[latin1]{inputenc}

\newtheorem{bigthm}{Theorem}   

\renewcommand{\emptyset}{\varnothing}

\newcommand{\sumtwo}[2]{\sum_{\substack{#1 \\ #2}}} 
\newcommand{\abs}[1]{\left| #1\right|}
\newcommand{\calA}{\mathcal{A}}
\newcommand{\calB}{\mathcal{B}}
\newcommand{\calC}{\mathcal{C}}

\newcommand{\calE}{\mathcal{E}}

\newcommand{\calN}{\mathcal{N}}

\newcommand{\calT}{\mathcal{T}}

\newcommand{\calY}{\mathcal{Y}}

\newcommand{\frd}{\mathfrak{d}}

\newcommand{\frh}{\mathfrak{h}}

\newcommand{\frs}{\mathfrak{s}}
\newcommand{\frt}{\mathfrak{t}}

\newcommand{\frD}{\mathfrak{D}}

\newcommand{\bbA}{\mathbb{A}}
\newcommand{\bbB}{\mathbb{B}}

\newcommand{\bbE}{\mathbb{E}}

\newcommand{\bbN}{\mathbb{N}}

\newcommand{\bbP}{\mathbb{P}}

\newcommand{\bbR}{\mathbb{R}}
\newcommand{\bbS}{\mathbb{S}}

\newcommand{\bbZ}{\mathbb{Z}}

\newcommand{\sfe}{{\sf e}}
\newcommand{\sfp}{{\sf p}}
\newcommand{\sfX}{{\sf X}}
\newcommand{\sfD}{{\sf D}}
\newcommand{\sfT}{{\sf T}}

\newcommand{\ueta}{\underline{\eta}}

\newcommand{\Zd}{\bbZ^d}
\newcommand{\Rd}{\bbR^d}

\newcommand{\Ka}{{\mathbf K}}

\newcommand{\Ua}{{\mathbf U}}
\newcommand{\UaK}[1]{{\mathbf U}_{#1}}

\newcommand{\setof}[2]{\left\{#1 \,:\, #2 \right\}}

\newcommand{\Lpartial}{\partial_{\scriptscriptstyle\rm v}}

\newcommand{\lb}{\left(}
\newcommand{\rb}{\right)}
\newcommand{\lbr}{\left\{}
\newcommand{\rbr}{\right\}}

\newcommand{\Step}[1]{\textbf{S{\small TEP}\,#1}}

\newcommand{\sfb}{\mathsf b}
\newcommand{\sff}{\mathsf f}

\newcommand{\1}{\mathbf{1}}
\newcommand{\smo}[1]{{\mathrm o}\lb #1\rb }

\newcommand{\df}{\stackrel{\bigtriangleup}{=}}

\newcommand{\leqs}{\lesssim}            
\newcommand{\geqs}{\gtrsim}             

\newcommand{\be}{\begin{equation}}
\newcommand{\ee}{\end{equation}}

\begin{document}

\title{Self-Attractive Random Walks:\\ The Case of Critical Drifts}
\author{Dmitry Ioffe\inst{1}\thanks{Supported by the Israeli Science Foundation grant 817/09.} \and Yvan Velenik\inst{2}\thanks{Supported by the Swiss National Science Foundation.}}
\institute{Technion, \email{ieioffe@technion.ac.il} \and Universit\'e de Gen\`eve, \email{Yvan.Velenik@unige.ch}}

\maketitle

\begin{abstract}
Self-attractive random walks (polymers) undergo a phase transition in terms of the applied drift (force): If the drift is strong enough, then the walk is ballistic, whereas in the case of small drifts self-attraction wins and the walk  is sub-ballistic. We show that, in any dimension $d\geq 2$, this transition is of first order. In fact, we prove that the walk is already ballistic at critical drifts, and establish the corresponding LLN and CLT.
\end{abstract}
\section{Introduction and results}
\subsection{Random walks in attractive potentials}
For each nearest-neighbor path $\gamma = (\gamma (0), \dots ,\gamma (n))$
on $\Zd$, we define the following quantities:

\noindent
\emph{Length} $\sfT (\gamma ) \equiv |\gamma | \df n$ and
 \emph{displacement}  $\sfX (\gamma ) \df \gamma (n) -\gamma (0)$.

\noindent
\emph{Local times:} For a given site $x\in\Zd$, we set
\[
 \ell_ \gamma [x] \, \df \, \sum_{k=0}^{|\gamma |} \1_{\{ \gamma (k) =x\}} .
\]

\noindent
\emph{Potential}  $\Phi (\gamma )$:
In this paper, we concentrate on potentials $\Phi$ which depend only on
the local times of $\gamma$; specifically, potentials of the form
\be 
\label{Potential}
\Phi (\gamma )\, \df\, \sum_{x\in\Zd}\phi\lb \ell_\gamma [x] \rb.
\ee
To every $h\in\Rd$ and $\lambda \in\bbR$, we associate the grand-canonical weights defined by
\be
\label{hlWeights}
a_{h ,\lambda} (\gamma )\, \df\, 
{\rm e}^{-\Phi (\gamma ) + h\cdot \sfX (\gamma ) -\lambda {\sfT(\gamma)}} 
\lb \frac{1}{2d}\rb^{\abs{\gamma }}
\df
{\rm e}^{-\Phi (\gamma ) + h\cdot \sfX (\gamma ) -\lambda {\sfT(\gamma)}} \sfp_d (\gamma ), 
\ee
where $ h\cdot \sfX $ is the usual scalar product on $\Rd$. 
If either $h$ or $\lambda$ is equal to zero, the corresponding entry is dropped from
the notation.  The potential $\phi$ is assumed to satisfy $\phi(0)=0$ and $\phi(n)>0$ for all $n>0$, to be non-decreasing and to be attractive:

\noindent
{\bf (A)} For all $\ell,m\in\bbN$,  $\phi (\ell+m ) \leq \phi (\ell) +\phi (m)$.
\smallskip 

\noindent
Notice that this condition implies that $\lambda_0=\lim_{n\to\infty} \phi(n)/n$ exists. Redefining $\lambda$ as $\lambda+\lambda_0$, it follows that we can assume, without loss of generality, that $\phi$ is sub-linear:

\noindent
{\bf (SL)} $\lim_{n\to\infty}\phi (n) /n =0$.
\smallskip

\noindent
An important class of examples is given by annealed random walks in random potentials: Let 
$V$ be a random variable which takes
values in $[0, \infty ]$; then $\phi (\ell ) = -\log\bbE\lb {\rm e}^{-\ell V }\rb$. In that case, the assumption (SL) can be reformulated as $0\in {\rm supp}(V)$.

\smallskip
\noindent
The main object we explore in this paper is the canonical probability measure
\be  
\label{eq:An}
\bbA_h^n (\gamma ) \df \frac{1}{A_h^n} a_h (\gamma )\1_{\lbr {\sfT(\gamma)}=n\rbr},
\ee
where $A_h^n$ is a normalization constant.
The behavior of $\bbA_h^n$ depends on the size of the drift $h$. We shall distinguish between two cases:
\begin{definition}
 \label{def:ballistic}
Assume that there exists $v = v(h)\in\bbR^d$, such that
\be 
\label{eq:velocity}
\lim_{n\to\infty} \bbA_h^n \Bigl( \frac{\sfX}{n}\Bigr) = v. 
\ee
Then the model is said to be sub-ballistic if $v = 0$ and ballistic if $v\neq 0$. 
\end{definition}
It is not clear \textit{a priori} that the limit in~\eqref{eq:velocity} exists.
One of our results here asserts that this is indeed the case for any choice of $h$.

The works~\cite{Sznitman-book,Zerner} explain how one should try to quantify drifts $h$:
There exists a closed convex body $\Ka\subset\bbR^d$ with non-empty interior,
$0\in {\rm int}\lb \Ka\rb$, such that 

\noindent
$\triangleright$ $h\in {\rm int}\lb \Ka\rb$ implies that the model is sub-ballistic~\cite{Zerner,Flury-LD}.

\noindent
$\triangleright$ $h\not\in \Ka$ implies that the model is ballistic~\cite{Zerner,Flury-LD,IoffeVelenik-Annealed}.
\smallskip

\noindent
We shall call the corresponding drifts sub-critical and super-critical, respectively.
$\partial\Ka$ is the set of critical drifts. Two main issues we address here are (i) what is the order of the 
ballistic to sub-ballistic transition, i.e., whether $v(h)$ in~\eqref{eq:velocity} vanishes as
$h\in \Ka^c$ converges to $\partial \Ka$, and (ii) what are the properties of $\bbA_h^n$ at critical drifts
$h\in\partial\Ka$.  
\begin{bigthm}
\label{thm:A}
Consider dimensions $d\geq 2$ and assume {\rm (A)} and ${\rm (SL)}$. Let $h\in\partial \Ka$. Then there exists
$v = v(h)\neq  0$ such that~\eqref{eq:velocity} holds. Moreover, the
displacement $\sfX$ satisfies a law of large numbers: There exists $\epsilon_n\downarrow 0$ such that 
\be  
\label{eq:LLN}
\sum_n \bbA_h^n\Bigl( \Bigl| \frac{\sfX}{n} - v(h) \Bigr| >\epsilon_n \Bigr) <\infty. 
\ee
$\sfX$ also satisfies a CLT: There exists a positive definite covariance matrix $\Sigma = \Sigma (h)$
such that, under $\lbr \bbA_h^n\rbr$, 
\be  
\label{eq:CLT}
\frac{1}{\sqrt{n}}\lb \sfX - n v(h) \rb \Rightarrow \calN (0, \Sigma ).
\ee
Finally, 
the mean displacement $v = v (h)$ is  continuous  on the (closed)  set of critical and super-critical 
drifts $\bar\Ka^c = \partial \Ka\cup\Ka^{c}$. In particular $\min_{h\in\partial \Ka} \abs{v (h )}$ is well defined and 
strictly positive. 
\end{bigthm}
While concluding the work on this manuscript, we have learned about the very nice and completely
independent paper~\cite{KosyginaMountford}. The authors of~\cite{KosyginaMountford} employ
different methods and techniques and their results overlap with ours. In one dimension, which is
not addressed here at all, they develop a comprehensive analysis of critical drifts whenever potentials
are different from that of the discrete Wiener sausage,
$\phi(\ell)=\beta\1_{\{\ell\geq1\}}$.
In particular, they prove that, in dimension $d=1$, the critical model is ballistic,
which implies that the corresponding transition is of first order. Thus, the one-dimensional discrete sausage case
is the only instance of a higher order transition~\cite{Vermet,IV-BPS}.
In higher dimensions, $d\geq 2$, the results of~\cite{KosyginaMountford} do not imply existence of critical velocity in~\eqref{eq:velocity}. In fact, the authors work in the conjugate ensemble of crossing random walks (see
Subsection~\ref{sub:crossing} below). Morally, they show that  $\liminf_{n\to\infty} \bbA_h^n\lb \abs{\sfX}/{n}\rb >0$ whenever $h\in\partial\Ka$.  In particular, although they did not state it explicitly, their results imply a first order transition in the following sense: If a sequence $\lbr h_n\rbr$ of super-critical drifts is converging to $h\in\partial\Ka$, then $\liminf \abs{v(h_{{n}})}>0$.

\bigskip
The problem considered above possesses numerous physical interpretations, such as stretched polymers, or magnetic flux lines in superconductors; see, e.g., \cite{MehraGrassberger}. In the polymer interpretation, our results say that a (homo)polymer with attraction between overlapping monomers undergoes a first order phase transition from a collapsed phase to an extended phase, as the opposite stretching forces acting at its endpoints cross a critical (direction-dependent) threshold, provided that the spatial dimension is at least $2$. That this phase transition is of first order was first predicted in~\cite{MehraGrassberger}, on the basis of numerical simulations, in the case of the discrete Wiener sausage. Our results fully confirm this conjecture and extend it to general self-attractive polymers; we also prove that the polymer is stretched at the phase transition point, with a well-defined macroscopic extension (and Gaussian fluctuations).

The corresponding problem when excluded-volume effects are also taken into account (say, by modeling the polymer path by a self-avoiding walk, with attraction between spatially neighboring monomers) is of major physical relevance, but remains beyond the reach of current rigorous techniques. It should be noted though that the phase transition is also expected to be of first order in that case when $d\geq 3$, but seems to be of second order when $d=2$~\cite{GrassbergerHsu}.

\subsection{Structure of the paper}
\label{sub:structure}
In Subsection~\ref{sub:crossing} below, we define the basic macroscopic quantity, the inverse correlation
length, and formulate two additional results closely related to Theorem~\ref{thm:A}: The geometry
of the critical set $\Ka$ is described in Theorem~\ref{thm:B}, and a reformulation of Theorem~\ref{thm:A} 
for the conjugate ensemble of crossing random walks is given in Theorem~\ref{thm:C}. 

The new coarse-graining procedure which lies in the heart of our approach is developed in
Section~\ref{sec:CG}. The eventual output is a large finite scale renewal type description of 
typical critical polymers as  formulated in Theorem~\ref{thm:renewal}, 
which sets up the stage for the proofs of our main results in Section~\ref{sec:Proofs}. 

\medskip
For convenience, we collect a number of elementary results about simple random walks in Appendix~\ref{app:RW}.

\subsection{Some notations}
Given two indexed sequences $\lbr f_\alpha\rbr$
and $\lbr g_\alpha\rbr$ of positive numbers, we say that $f_\alpha\leqs g_\alpha$ \emph{uniformly} in
$\alpha$ if there exists a constant $c<\infty$ such that $f_\alpha \leq cg_\alpha$ \emph{for all}
indices $\alpha$. We shall also write $f_\alpha\simeq g_\alpha$ if both $f_\alpha\leqs g_\alpha$ and $g_\alpha\leqs f_\alpha$.

We shall denote by $\abs{x}$ the Euclidean norm of $x\in\bbR^d$ and by $\|x\|\df \sum_1^d \abs{x_i}$ its $\ell_1$-norm.

We write $x\sim y$ when $x,y$ are neighboring vertices of $\Zd$.

If $B\subset\bbR^d$, we shall write $\Lpartial B\df\setof{z\not\in B}{\exists y\in B\,:\, y\sim z}$ for its external boundary when seen as a subset of $\Zd$.

Finally, if $\gamma=(\gamma(0),\ldots,\gamma(n))$ is some path, then we denote by $\mathring{\gamma}$ the path with its final endpoint removed: $\mathring{\gamma} \df (\gamma(0),\ldots,\gamma(n-1))$.

\subsection{Conjugate ensembles, crossing random walks and the set $\Ka$}
\label{sub:crossing}
\paragraph{Inverse correlation length $\xi$ at criticality} 
Recall that we drop the subscripts $h$ or $\lambda$ from the notation  
$a_{h, \lambda}$ whenever the corresponding parameter equals zero.
The critical two-point function $A(x)$ is defined by
\be 
\label{eq:two-point}
 A (x) \df \sum_{\gamma : 0\to x} a (\gamma ) .
\ee
The following observation is presumably well understood. We put it here for completeness. 
\begin{lemma}
In the attractive case, the critical two-point function $ A (x)$ is finite  in 
any dimension $d\geq 1$. Consequently, 
\[
 \xi (x)\, \df \, -\lim_{n\to\infty}\frac1{n}\log A (\lfloor nx\rfloor ) 
\]
exists and is a
norm on $\bbR^d$. Moreover,
\be 
\label{eq:Additive}
A (x) \leq {\rm e}^{-\xi  (x) } , 
\ee
for all $x\in\bbZ^d$. 
\end{lemma}
\begin{proof}
Since $A$ is super-multiplicative in the attractive case, it is sufficient to prove that $A(x)$ is finite. 
Let $\Lambda_k$ be the $d_1$-ball around the origin,
$$
\Lambda_k \df \setof{y\in\bbZ^d}{\| y\|< k}.$$
Let us split  the family of paths $\calA_x \df \setof{\gamma}{0\to x}$ as
\[
 \calA_x= \bigcup_k\calA_x^{(k)} ,
\]
where
\[
 \calA_x^{(k)} \df \setof{\gamma\in\calA_x}{\gamma\cap\Lambda_{k\|x\|}^c\neq\emptyset
\quad{\rm but}\quad
\gamma\cap\Lambda_{(k+1) \|x\| }^c =\emptyset}.
\]
For $\gamma\in \calA_x^{(k)}$, $\Phi (\gamma ) \geq k\|x\|\phi (1)$.  Accordingly, 
\[  
a (\gamma )  = {\rm e}^{-\Phi (\gamma )} \sfp_d (\gamma)
\leq {\rm e}^{- k\|x\|\phi (1)} \sfp_d (\gamma).
\]
For $\Lambda \subset \bbZ^d$, let $G_\Lambda $ be  the Green function of the simple random walk (SRW) with Dirichlet boundary conditions at the boundary of $\Lambda$:
\[
G_\Lambda(0,x) \df E_{\rm SRW} \bigl[ \sum_{n=0}^{\tau_\Lambda-1} \1_{\{S_n=x\}} \bigr],
\]
where $(S_n)_{n\geq 0}$ is a simple random walk starting at $0$, $E_{\rm SRW}$ the corresponding expectation, and $\tau_\Lambda = \inf\{n\geq 0\,:\, S_n\not\in\Lambda\}$.

\noindent
It follows that, see~\eqref{eq:GFLambda},
\begin{align}
A^{(k)} (x) &\df 
\sum_{\gamma\in\calA_x^{(k)}} a (\gamma )\leq {\rm e}^{- k\|x\|\phi (1)} G_{\Lambda_{(k+1)\|x\|}}
(0,x )\nonumber\\
&\leqs
\begin{cases}
 {\rm e}^{- k\|x\|\phi (1)} (k\|x\|)^2\qquad &\text{for $d=1$,}\\
{\rm e}^{- k\|x\|\phi (1)}\log(k\|x\|) &\text{for $d=2$,}\\
{\rm e}^{- k\|x\|\phi (1)} &\text{for $d\geq 3$,}
\end{cases}
\label{eq:Dlnotk}
\end{align}
uniformly in $x$ and $k$.
Consequently $A (x) = \sum_{k\geq 1} 
A^{(k)} (x)$ converges in any dimension (and is exponentially decreasing in $x$).
\qed
\end{proof}
\paragraph{The set $\Ka$ and related convex geometry}
\begin{figure}[t]
\begin{center}
\scalebox{.4}{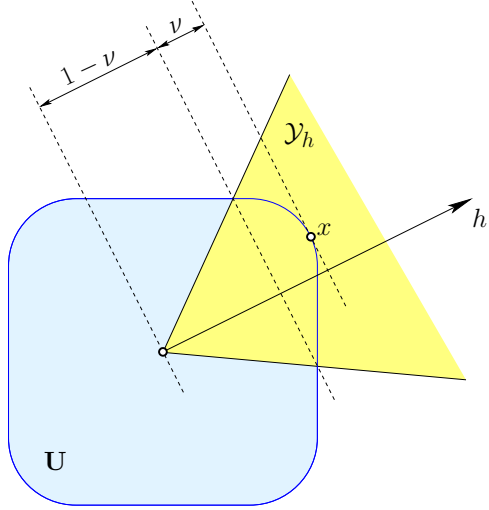}
\end{center}
\caption{The point $x\in\partial\Ua$ conjugate to the vector $h$, and the cone $\calY_h$ for a given value of $\nu$.}
\label{fig:cone}
\end{figure}
Let us denote by $\UaK{r}(x)$ the ball of radius $r$ in the $\xi$-metric centered at $x$, and let us set $\UaK{r}\equiv\UaK{r}(0)$ and $\Ua\equiv \UaK{1}$.

The set (of sub-critical and critical 
drifts) $\Ka$ is defined through $\xi$: Namely, $\xi$ is the support function
of $\Ka$, 
\be 
\label{eq:support}
\xi (x) = \max_{h\in\partial\Ka} h\cdot x .
\ee
Equivalently,
\[
\Ka = \bigl\{ h\in\bbR^d \,:\, h\cdot x \leq \xi(x)\; \forall x\in\bbR^d \bigr\} = \bigl\{ h\in\bbR^d \,:\, h\cdot x \leq 1\; \forall x\in\Ua \bigr\}.
\]
Note that the unit ball $\Ua$ is the polar set of $\Ka$: $\Ua \df \setof{x}{x\cdot h\leq 1,\, \forall h\in\Ka}$.
In particular, if $h\in\partial \Ka$ and $x\neq 0$, then 
\be 
\label{eq:duality}
\xi (x) = h\cdot x \Longleftrightarrow \lbr\text{$h$ is orthogonal to a supporting  hyperplane to 
$\partial\UaK{\xi (x)}$
at $x$}\rbr .
\ee
When this is the case, we say that $h$ and $x$ are conjugate to each other.

Let us fix $\nu\in (0,1 )$. For $h\in\partial\Ka$, we introduce the forward cone, see Fig.~\ref{fig:cone},
\be
\label{eq:cone}
\calY_h \df \lbr x\in\bbR^d ~:~ \frs_h (x)\df \xi (x) -h\cdot x \leq \nu \xi (x)\rbr . 
\ee
The surcharge function $\frs_h$  plays a crucial role in quantifying the large scale behavior
of typical paths $\gamma$ under $\bbA_h^n$. Observe that, by definition, $\frs_h(y) \geq 0$ for all $y\in\bbR^d$, and $\frs_h(x)=0$ if and only if $x$ is conjugate to $h$. The choice of $\nu$ in~\eqref{eq:cone} corresponds to a
sufficiently small value of $\kappa$ in the following lemma.
\begin{lemma}
\label{lem:nu} 
For every $\kappa >0$, there exists $\nu {= \nu (\kappa )}>0$ such that, uniformly in $h\in\partial\Ka$:
\[
x,y\in \partial\Ua\cap\calY_h \implies x+\kappa y\notin \Ua.
\]
\end{lemma}
\begin{proof}
Just notice that, by~\eqref{eq:support},
$\xi (x+\kappa y ) \geq h\cdot (x+\kappa y) \geq (1-\nu)(\xi(x)+\kappa\xi(y)) = (1-\nu)(1+\kappa )$. 
\qed
\end{proof}
Lemma~\ref{lem:nu} is a generic statement which holds for any norm on 
$\bbR^d$. In Subsection~\ref{sub:geometry}, we prove a much stronger result about
critical $\xi$ and  $\Ka$:
\begin{bigthm}
\label{thm:B}
In any dimension $d\geq 2$, the assumptions (A) and (SL) imply that the boundary $\partial\Ka$ 
of the critical shape $\Ka$ is locally analytic and has a uniformly strictly positive Gaussian curvature.
\end{bigthm}
\paragraph{Conjugate directions and crossing random walks} 
Let $d\geq 2$ and  $x\in\bbZ^d$. The ensemble of crossing random walks is given by the probability measure
\be 
\label{eq:Ax}
\bbA_x (\gamma ) \df \frac{a(\gamma )}{A (x) }\1_{\lbr \sfX (\gamma )= x\rbr } .
\ee
This ensemble is conjugate to~\eqref{eq:An} through the relation $\xi (x) = h\cdot x$, which, 
by Theorem~\ref{thm:B}, is unambiguous. That is, for any $x\in\bbR^d\setminus \{0\}$, there exists a
unique $h\in\partial \Ka$ such that $\xi (x) = h\cdot x$.
When $x$ is conjugate to $h$ and such that $\abs{x}/n = v(h)$, there is equivalence (in the sense of statistical mechanics) between the two ensembles $\bbA_h^n$ and $\bbA_x$.
\begin{bigthm}
\label{thm:C}
Let $\abs{x_m}\to\infty$ in  such a fashion that $x_m /\abs{x_m}\to g\in\bbS^{d-1}$. Let $h\in\partial\Ka$
be the conjugate vector to $g$. Then,
\begin{equation}
\label{eq:C}
\lim_{m\to\infty}\frac1{\abs{x_m }}\bbA_{x_m}(\sfT)  = \frac1{\abs{v (h )}} .
\end{equation}
Moreover, both the LLN and the CLT hold. In particular,  there exists $\sigma^2 = \sigma ^2 (h) >0$, such that, under $\lbr\bbA_{x_m}\rbr$, 
\be 
\label{eq:TCLT}
\frac{1}{\sqrt{\abs{x_m }}}\bigl( \sfT - \bbA_{x_m}(\sfT)\bigr) \Rightarrow \calN (0, \sigma^2 ) .
\ee
\end{bigthm}

\subsection*{Acknowledgments} We are very grateful to anonymous referees for a very careful reading of our text and for making several suggestions that have certainly improved its readability.

\section{Critical Coarse-Graining}
\label{sec:CG}
Let us try to explain the difference between the super-critical case $h\not\in\Ka$ and 
the critical case $h\in\partial \Ka$. If $h\not\in\Ka$, then there exists $\lambda >0$ 
(see~\cite{IoffeVelenik-Annealed}) such that
the weights $a_h$ are closely related to the weights $a_\lambda$. In fact $\lambda$ is an 
appropriately chosen parameter in the conjugate ensemble of crossing random walks.  This means that
the model at super-critical drifts $h$ is essentially massive. Existence of positive mass or, 
in terms of random walks, strictly positive killing potential, has a huge impact on the 
corresponding 
sample path properties and, certainly, facilitates proofs and arguments. 

The critical case corresponds to $\lambda =0$, and the  main thrust of 
Proposition~\ref{prop:skeletons} below is to enable a control of local time geometry of
random walks at zero mass using only  positivity of  sub-linear interaction potentials. 
Subsection~\ref{sub:CG-micro} studies implications for the large scale renewal structure of 
microscopic polymers, as summarized in Theorem~\ref{thm:renewal}. 
\subsection{Decorated skeletons} 
\label{sub:deco}
The decorated skeleton $\hat{\gamma}_K$ of a path $\gamma$ will be  eventually defined in~\eqref{eq:gammaK}. It consists of a trunk, hairs and attached dirty boxes. The whole construction relies on a scale parameter $K$ that will be fixed (large enough, but independently of $\abs{x}$) later on. The construction involves several steps.

\paragraph{Trunks and Pre-hairs} 
Fix $h\in\partial \Ka$.  
Our construction is a modification of the coarse-graining employed in~\cite{IoffeVelenik-Annealed}: As in the super-critical
case, we represent a path $\gamma$ as a concatenation,
\be 
\label{gleta}
 \gamma\, =\, \gamma_0\cup\eta_0\cup\gamma_1\cup\dots\cup \eta_{m -1}\cup \gamma_m .
\ee
Let us recall the corresponding construction (see Fig.~\ref{fig:coarsegraining}):  We start by defining the trunk $\frt_K$ of $\gamma$.

\smallskip
\par
\begingroup
\leftskip3em
\rightskip1em
\rightskip\leftskip
\noindent
\Step{ 0}. Set  $u_0 =0$, $\tau_0 = 0$ and $\frt_K =\lbr u_0\rbr$.  Go to
\Step{ 1}. 
\smallskip

\noindent
\Step{ $\boldsymbol{\ell+1}$}.
If $\lb\gamma (\tau_\ell ) ,\dots \gamma (n)\rb\subseteq \UaK{K} (u_\ell  )$, then 
set $\sigma_{\ell } = n$ and stop.
Otherwise, define
\begin{align*}
&\sigma_{\ell } = \min\lbr i >\tau_\ell ~:~ 
\gamma (i)\not\in \UaK{K} (u_\ell  )\rbr\\
&\qquad\text{ and}\\
&\tau_{\ell +1} = 1+ \max 
\lbr i >\tau_\ell ~:~ 
\gamma (i) \in \UaK{K} (u_\ell )\rbr .
\end{align*}
Set $v_{\ell} = \gamma (\sigma_{\ell})$ and $u_{\ell +1} = \gamma (\tau_{\ell+1})$. 
Update $\frt_K = \frt_K\cup\lbr u_{\ell +1}\rbr$ and go to \Step{$\boldsymbol{\ell+2}$}. \hfill$\blacksquare$
\par
\endgroup

\medskip
\noindent
Apart from producing $\frt_K$ the above algorithm leads to a decomposition
of $\gamma$ as in~\eqref{gleta} with
\[
\gamma_{\ell }\, =\, \lb \gamma(\tau_{\ell } ),\dots ,\gamma (\sigma_{\ell } )\rb\quad
\text{and}\quad \eta_{\ell} = \lb \gamma (\sigma_{\ell}),\dots ,\gamma (\tau_{\ell +1} )\rb .
\]
By construction, $\gamma_{j}\cap  \UaK{K} (u_\ell ) = \emptyset$ and 
$\eta_{j}\cap  \UaK{K} (u_\ell ) = \emptyset$ for any $j >\ell$. Moreover, the trunk
$\frt_K = \lbr u_0 , \dots ,u_m\rbr$ is well separated in the following sense: 
There exists $c \in (0,\infty )$
such that $d_1 ( u_j ,u_\ell )\geq c^{-1} K$ for all $\ell\neq j$ and 
$d_1 (u_\ell ,u_{\ell+1} )\leq cK$ for all $\ell$. 

The disjointness of the paths $\gamma_i$ imply that $\Phi(\gamma)\geq \sum_i \Phi(\gamma_i)$. This will be crucial in the sequel and will allow a rather straightforward control of these paths. The paths $\eta_i$, on the other hand, are a nuisance and their control is the main difficulty in the critical case.

The hairs $\frh_K$ of the decorated skeleton $\hat{\gamma}_K$, which will be introduced below, take into account those paths $\eta_\ell$ which are long on the scale $K$; their construction relies on that of pre-hairs, to which we turn now. Recall that $\eta_\ell : v_\ell \mapsto u_{\ell+1}$.
Since we are eventually going to fix $u_\ell$-s and not $v_\ell$-s 
it is more convenient to think about $\eta_\ell$ as of a reversed path from $u_{\ell +1}$ to $v_\ell$. 
Then the $\ell$-th \emph{pre}-hair $\tilde \frh_K^\ell$ of $\gamma$ is constructed as follows:
If $\eta_\ell\subseteq \UaK{K} (u_{\ell+1})$ then $\tilde \frh_K^\ell =\emptyset$. Otherwise, 
construct the polygonal approximation $\tilde \frh_K^\ell$  to $\eta_\ell$ 
following exactly the same rules as in 
the construction of  $K$-skeletons in the repulsive case in~\cite{IoffeVelenik-Annealed}. Namely, for any path $\eta = \lb \eta (0), \dots , \eta (m)\rb $ such that $\eta\not\subseteq \UaK{K} (\eta (0))$, we define (see Fig.~\ref{fig:coarsegraining}):

\smallskip
\par
\begingroup
\leftskip3em
\rightskip1em
\rightskip\leftskip
\noindent
\Step{0}. Set $z_0=\eta (0 )$, $\tau_0 = 0$ and $\tilde\frh_K ={\emptyset}$. Go
to \Step{1}. 
\smallskip

\noindent
\Step{$\boldsymbol{r+1}$}.  If $\lb \eta  (\tau_r ),\dots ,\eta  (m)\rb \subseteq \UaK{K} (z_r  )$,
then stop. Otherwise set
\[
\tau_{r+1}\, =\, \min\lbr j> \tau_r~:~ \eta  (j)\not\in \UaK{K} (z_r )\rbr .
\]
Define $z_{r+1} = \eta  (\tau_{r+1} )$, update 
$\tilde\frh_K = \tilde\frh_K \cup \lbr z_{r+1}\rbr$ and go to 
\Step{$\boldsymbol{r+2}$}. \hfill$\blacksquare$
\par
\endgroup
\begin{figure}[t]
\begin{center}
\scalebox{0.8}{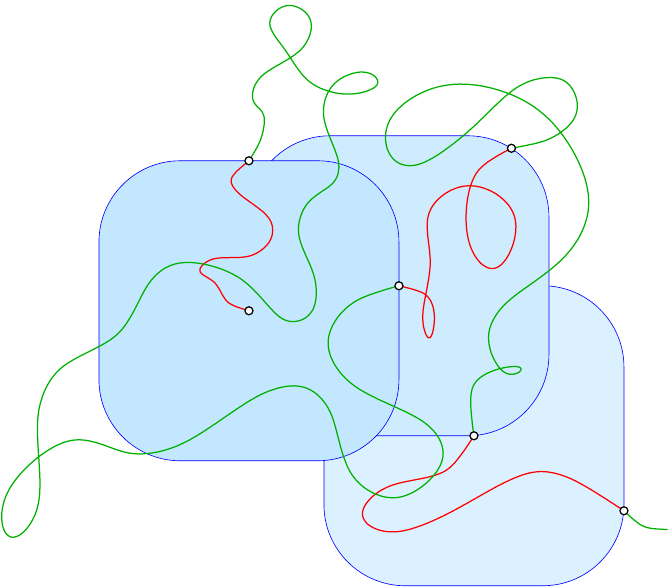}
\hspace*{4mm}
\scalebox{0.8}{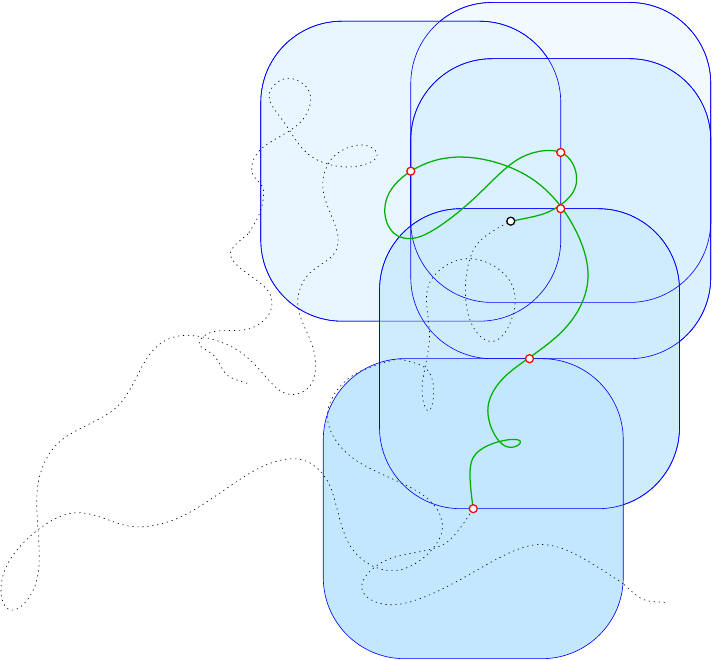}
\end{center}
\caption{\textbf{Left:} construction of the vertices $u_k,v_k$ associated to a path
$\gamma$ (only the beginning of the path is shown). Informally, $v_k$ is the first vertex of the path after $u_k$ lying outside of $\UaK{K}(u_k)$ and $u_{k+1}$ is the vertex following the last visit of $\gamma$ inside $\UaK{K}(u_k)$. \textbf{Right:} construction of the pre-hair $\tilde \frh_K^1= \lbr z_1, z_2, z_3 ,z_4\rbr$ associated to the reversed path $\eta_1:u_2\to v_1$.}
\label{fig:coarsegraining}
\end{figure}

\paragraph{Basic Estimate at Criticality} 
At this stage we need to proceed with more care than in the case of super-critical drifts considered in~\cite{IoffeVelenik-Annealed}: There, the paths $\eta_\ell$ were controlled by comparison with a simple random 
walk killed at rate $\lambda$, or, in other words, the lifetimes were geometric with $p= 1- e^{-\lambda}$. This would not work in the critical case of $\lambda =0$. Indeed, for $z\in\Lpartial \UaK{K}$, we have, see~\eqref{eq:EssentiallyUniformExitProb},
\be  
\label{eq:ExitBall}
\sumtwo{\eta :0\to z}{\mathring{\eta}\subseteq \UaK{K}} \sfp_d  (\eta )
\simeq \frac{1}{K^{d-1}} .
\ee
In order to recover exponential decay, we need to tighten control over the range of the various paths in
question. Morally, the idea is to decompose such a path into three types of pieces: those that remain close to the trunk (and thus suffer an entropic loss), those that wander far away from the trunk but have a small range (and thus also suffer from entropic loss), and those that wander far away from the trunk and have a fat range (and thus pay a high energetic price).

\medskip
Given a path $\eta$, let us denote its range by $R[\eta ] \df \setof{z}{z= \eta (j)\ \text{for some $j$}}$. The control of pieces with small range is done through the following lemma.
\begin{lemma}
 \label{lem:range}
For every $c_1 >0$ there exists $c_2 >0$ such that
\be  
\label{eq:range}
\sumtwo{\eta :0\to z}{\mathring{\eta}\subseteq \UaK{K}} 
\sfp_d (\eta )
\1_{\lbr\abs{R[\eta ]}\leq c_1 K \rbr} \leq {\rm e}^{-c_2 K} ,
\ee
for all large enough $K$ and all $z\in\Lpartial \UaK{K}$.  
\end{lemma}
\begin{proof}
The worst case is when the dimension $d=2$. We treat all the dimensions 
simultaneously and, consequently, 
some of the estimates we obtain
may be further improved  in the transient case $d\geq 3$.
 
Let $\tau_K$ be the first exit time of a simple random walk from $\UaK{K}$. Set
$R_K \df R\left[ \eta[0, \tau_K ]\right]$.  We claim that, for any $C>0$, it is possible to find
$\nu >0$ such that 
\be  
\label{eq:RnuBound}
E_{\rm SRW} {\rm e}^{-\nu \abs{R_K }} \leq {\rm e}^{-C\nu K} ,
\ee
for all sufficiently large scale $K$. Note that~\eqref{eq:RnuBound} instantly implies the claim of the
lemma. Indeed, the left-hand side of~\eqref{eq:range}
is bounded above by
$P_{\rm SRW}\lb \abs{R_K} \leq c_1 K\rb$. However, in view of~\eqref{eq:RnuBound},
\[
P_{\rm SRW}\lb \abs{R_K} \leq c_1 K\rb 
\leq {\rm e}^{c_1\nu K}E_{\rm SRW} {\rm e}^{-\nu \abs{R_K }} \leq {\rm e}^{-(C -c_1 )\nu K}.
\]
It remains to verify~\eqref{eq:RnuBound}. This will be done in two steps. First of all, we claim that
for any $C^\prime>0$ it is possible to find $\nu >0$ and a scale $N$ such that 
\be 
\label{eq:HalfBound}
 E_{{\rm SRW}} {\rm e}^{-\nu \abs{R_{2N}\cap\UaK{N} }} \leq {\rm e}^{-C^\prime\nu N} .
\ee
Indeed, expanding both sides in a Taylor series in $\nu$ at fixed $N$, and using the inequality $E_{\rm SRW}^0 \abs{R_{2N}\cap\UaK{N}}\geq c_3N^2/\log N$ (see~\eqref{eq:LowerBoundRange}), \eqref{eq:HalfBound} follows as soon as we choose $\nu$ sufficiently small (depending on $N$ and $C'$).

\begin{figure}[t]
\begin{center}
\scalebox{0.5}{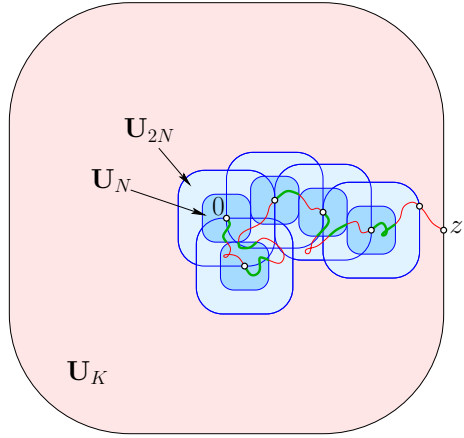}
\end{center}
\caption{The decomposition of the path $\eta :0\to z$ in the proof of Lemma~\ref{lem:range}. Observe that, by construction, the bold pieces of the path  belonging to different balls  are disjoint.}
\label{fig:exit}
\end{figure}
Our next step is to iterate~\eqref{eq:HalfBound} to all large scales; see Fig.~\ref{fig:exit}. Let $K >N$. Consider the following
sequence of stopping times: $\sigma_0 = 0$ and 
\[
 \sigma_n =\min\Bigl\{t >\sigma_{n-1}\,:\,\eta (t )
\not\in\bigcup_{\ell=0}^{n-1} \UaK{2N} \lb \eta (\sigma_\ell )\rb\Bigr\}.
\]
Then $\sigma_n \leq \tau_K$ as long as $n \leq K/2 (N+\varsigma)$,  where 
\[
\varsigma\df \sup_N\max_{y\in\Lpartial \UaK{N}}\lb \xi (y) -N\rb . 
\]
Let $\theta_{\sigma_\ell}\eta $ be the shifted path (which starts at $\eta (\sigma_\ell )$).
Since by construction the balls $\UaK{N} (\eta (\sigma_\ell ))$ are disjoint, 
the sets $  R_{2N}[\theta_{\sigma_\ell}\eta ]
 \cap \UaK{N}  (\eta (\sigma_\ell )) $ are also disjoint.
Hence, for any $n\leq K/2(N+\varsigma)$, 
\[
 \abs{R_K[\eta ]} \geq \sum_{\ell =0}^{n-1} \abs{R_{2N}[\theta_{\sigma_\ell}\eta]
 \cap \UaK{N}(\eta (\sigma_\ell ))} .
\]
Therefore, by the strong Markov property, 
\[
 E_{\rm SRW} {\rm e}^{-\nu \abs{R_K }} \leq  
\lb E_{\rm SRW} {\rm e}^{-\nu \abs{R_{2N}\cap\UaK{N} }}\rb^{{\lfloor K/(2N+\varsigma)\rfloor}} ,
\]
and, in view of~\eqref{eq:HalfBound}, \eqref{eq:RnuBound} follows.
\qed
\end{proof}

\paragraph{Dirty Boxes}
The point is that the potential $\sum_{x}\phi_\beta \lb \ell_\gamma (x )\rb \geq \phi_\beta (1) 
\abs{R[\gamma ]}$ suppresses paths $\eta_\ell$ with fat ranges.  Let us record this observation as follows:
Consider a partition of $\Zd$ into disjoint cubic boxes of sidelength $K$, 
\be   
\label{eq:boxPartition}
 \Zd = \bigvee_p B_K (p ) , \quad p\in K\Zd .
\ee
Given the decomposition~\eqref{gleta}, let us say that $p\in K\Zd$ is dirty, $p\in\frd_K$, if
\be  
\label{eq:dirty}
 \abs{ R[\ueta ] \cap B_K (p ) } \geq 
c_1 K , 
\ee 
where $R[\ueta ]\df \cup R[ \eta_\ell ]$. 

\paragraph{Hairs} 
Hairs $\frh^\ell_K\subseteq \tilde\frh_K^\ell$ contain only clean steps: 
Define 
\[
 \frD_K \df \bigcup_{p\in\frd_K}  B_K (p )  .
\]
By definition,
\[
 \frh^\ell_K \df \setof{z\in \tilde\frh_K^\ell}{
\UaK{{2K}}
 (z )\cap  \frD_K = \emptyset} \quad\text{and} \quad
\frh_K \df \lbr \frh_K^1 ,\dots ,\frh_K^n\rbr .
\]
Since there exists $\alpha <\infty$ such that any ball $\UaK{K} (u)$ intersects at most $\alpha$ 
disjoint boxes 
$B_K (p )$, 
the definition is set up in such a way that any path $\eta$ contributing to some hair $\frh^\ell_K$ (i.e., the paths $\lambda_\ell^k:z_i^\ell\to w_i^\ell$ of Fig.~\ref{fig:prehairs}) automatically
satisfies the constraint $\abs{R [\eta ]}\leq c_1 \alpha K$.
\smallskip

\paragraph{Decorated skeletons} 
The decorated skeleton $\hat\gamma_K$ of $\gamma$ is the collection
\be  
\label{eq:gammaK}
\hat\gamma_K =\lbr \frt_K , \frd_K , \frh_K\rbr .
\ee
\paragraph{Upper Bound on the Weight of a Decorated Skeleton} 
The following  statement is crucial for our control of the geometry of
decorated skeletons:
\begin{proposition}
 \label{lem:skeleton}
For any $c_1$ large enough in~\eqref{eq:dirty}, there exists $c_4>0$ such that, on all large enough scales $K$,
\be  
\label{eq:skeleton}
A \lb \hat\gamma_K|\; x \rb \df \sumtwo{\gamma\sim \hat\gamma_K}{X(\gamma ) = x} a(\gamma ) 
\leqs
{\rm e}^{-
K \lb   \abs{\frt_K} + c_2 \sum_\ell \abs{\frh_K^\ell} +c_4 \abs{\frd_K}\rb \lb 1 -\mathrm{o}_{\scriptscriptstyle K}(1)\rb} ,
\ee
uniformly in $x\in\bbZ^d$ and in $\hat\gamma_K$; the constant $c_2$ was introduced in Lemma~\ref{lem:range}, $\lim_{K\to\infty}\mathrm{o}_{\scriptscriptstyle K}(1)=0$ uniformly in all other parameters, and the notation $\gamma\sim\hat\gamma_K$ means that the path $\gamma$ gives rise to the skeleton $\hat\gamma_K$.
\end{proposition}
\begin{remark}
 Note that $A \lb \cdot \; |\; x \rb$ in 
\eqref{eq:skeleton}
is a {\em restricted} partition function, not a conditional one.
\end{remark}
\begin{proof}
Let us fix a decorated skeleton $\hat\gamma_K$.  We need to derive an upper bound on 
\linebreak
$
\sum_{\gamma\sim\hat\gamma_K} {\rm e}^{-\Phi (\gamma )} p(\gamma ) .
$ 
Since the paths $\gamma_i$ in~\eqref{gleta} are disjoint,
\begin{align}
\Phi  (\gamma ) &= \sum_x \phi  (\ell_\gamma (x) ) \geq c_1\phi  (1 ) K
\abs{\frd_K } + \sum_{x\not\in \frD_K} \phi  (\ell_\gamma (x) ) \nonumber\\
&\geq c_1\phi  (1 ) K
\abs{\frd_K } + \sum_{\ell =0}^{m-1} \1_{\lbr \UaK{K}  (u_\ell )\cap\frD_K =\emptyset\rbr}
\Phi  (\gamma_\ell ) ,
\label{eq:PotSplit}
\end{align}
for any $\gamma= \gamma_0\cup\eta_0\cup\gamma_1\cup\dots\cup \eta_{m -1}\cup \gamma_m \sim \hat\gamma_K$.  It follows that
\begin{multline*}
\sum_{\gamma\sim\hat\gamma_K} {\rm e}^{-\Phi  (\gamma )} 
\sfp_d (\gamma )\\
\leq 
{\rm e}^{- 
c_1\phi (1 ) K
\abs{\frd_K }} 
\sum_{\gamma\sim\hat\gamma_K}
\sfp_d (\gamma )\prod_\ell \exp\bigl(- \1_{\lbr \UaK{K} (u_\ell )\cap\frD_K =\emptyset\rbr}
\Phi  (\gamma_\ell)\bigr) .
\end{multline*}
Consider first the trunk 
$\frt_K = \lbr u_1, u_2, \dots ,u_n\rbr$ and the corresponding 
contribution of the paths $\gamma_\ell$. 
 If $\UaK{K} (u_\ell ) \cap \frD_K =\emptyset$, then the latter is bounded by
\[
\sum_{\gamma  :u_\ell \to \Lpartial \UaK{K} (u_\ell )}
\sfp_d (\gamma )
{\rm e}^{-\Phi  (\gamma )}  \leqs K^{d-1}{\rm e}^{-K},
\]
as can be seen from~\eqref{eq:Additive} and the fact that, by construction, $\xi(v-u_\ell)> K$ for all $v\in\Lpartial\UaK{K} (u_\ell)$.
Otherwise, if $\UaK{K} (u_\ell ) \cap \frD_K \neq \emptyset$, 
we need to take into account interaction with hairs and, \textit{a priori}, 
proceeding exactly as in~\eqref{eq:ExitBall} only bounds the expression by a constant.
However, by the spatial separation property of trunks, 
any box $B_K (p )$ intersects at most $\alpha^\prime $ balls of the type $\UaK{K} (u_\ell )$, 
regardless of the particular trunk we consider.  Since we 
are entitled to choose $c_1$ in~\eqref{eq:dirty}
arbitrary large, we can choose it of the form $c_1 = \lb c_4 +\alpha^\prime \rb /\phi  (1)$. With such a choice, the first term in the right-hand side of~\eqref{eq:PotSplit} becomes
\[
c_4 K \abs{\frd_K} + K \alpha'\abs{\frd_K} \geq c_4 K \abs{\frd_K} + K \#\setof{\ell}{\UaK{K} (u_\ell )\cap\frD_K \neq\emptyset},
\]
and we recover both the original ${\rm e}^{-K}$ price per each step of the trunk and the interaction potential 
price ${\rm e}^{-c_4 K}$ for each dirty box. 

Hairs are controlled through Lemma~\ref{lem:range}, 
which yields a contribution $e^{-c_2K|\frh_K|}$.
It remains to sum out the weights of all dirty 
parts of all the pre-hairs which are compatible with $\hat\gamma_K =\lbr \frt_K , \frd_K , \frh_K\rbr$. 
Let us focus on the worst case of the recurrent dimension $d=2$.  Note first of all that, by the strong Markov property,
\be  
\label{eq:ThinReturns}
\sum_{\eta : 0\to z} \sfp_d (\eta  )\1_{\lbr\abs{ R[\eta ]\cap \UaK{K}(z)}\leq c_1 K\rbr} \leq
\sum_{\eta : 0\to 0} \sfp_d (\eta  )\1_{\lbr\abs{ R[\eta ]\cap \UaK{K}}\leq c_1 K\rbr} \leqs \log K .
\ee
Indeed, in view of Lemma~\ref{lem:range} and by the strong Markov property,
\begin{align*}
\sum_{\eta : 0\to 0}
\sfp_d (\eta  )&\1_{\lbr\abs{ R[\eta ]\cap \UaK{K}}\leq c_1 K\rbr}\\
&=
\sumtwo{\eta : 0\to 0}{\eta\not \subseteq \UaK{K}} \sfp_d (\eta  )
\1_{\lbr\abs{ R[\eta ]\cap \UaK{K}}\leq c_1 K\rbr} + 
\sumtwo{\eta : 0\to 0}{\eta\subseteq \UaK{K}} 
\sfp_d (\eta  )
\1_{\lbr\abs{ R[\eta ]\cap \UaK{K}}\leq c_1 K\rbr}\\
&
\leqs K^{d-1}{{\rm e}^{- c_2 K}}
 \sum_{\eta : 0\to 0} 
\sfp_d (\eta  )\1_{\lbr\abs{ R[\eta ]\cap \UaK{K}}\leq c_1 K\rbr}
+
\sumtwo{\eta : 0\to 0}{\eta\subseteq \UaK{K}} 
\sfp_d (\eta  ),
\end{align*}
and thus
\[
\sum_{\eta : 0\to 0}
\sfp_d (\eta  )\1_{\lbr\abs{ R[\eta ]\cap \UaK{K}}\leq c_1 K\rbr}\\
\leqs \bigl(1-K^{d-1}{{\rm e}^{- c_2 K}}\bigr)^{-1}\, \sumtwo{\eta : 0\to 0}{\eta\subseteq \UaK{K}} 
\sfp_d (\eta).
\]
Therefore, the main contribution to the right-hand side of~\eqref{eq:ThinReturns} is at most of order $\log K$ as  it comes from paths $\eta$ confined inside $\UaK{K}$, see~\eqref{eq:GFLambda}.

\begin{figure}[t]
\begin{center}
\scalebox{0.6}{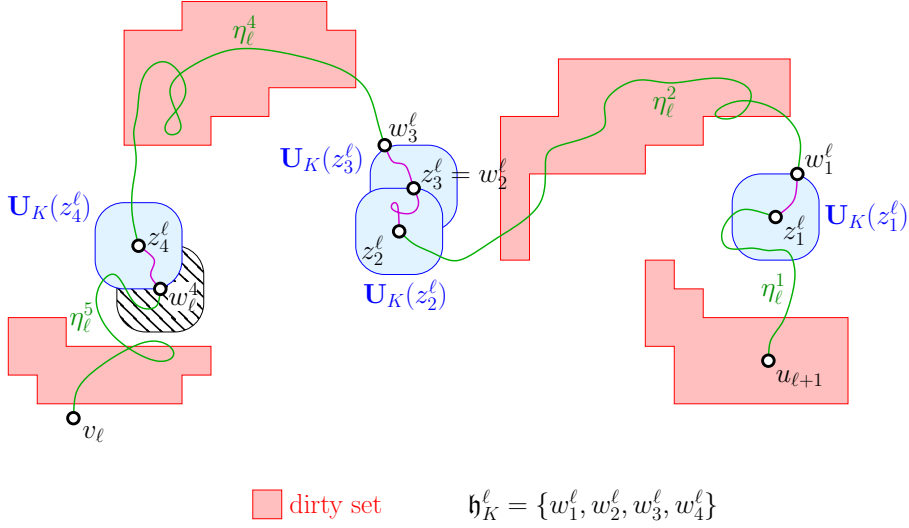}
\end{center}
\caption{The paths $\eta_\ell^1,\eta_\ell^2,\eta_\ell^3=\emptyset,\eta_\ell^4$ and $\eta_\ell^5$ account for compatible pre-hairs. Notice that the range of the path $\eta_\ell^5$ inside the ball $\UaK{K}(w_4^\ell)$ cannot exceed $c_1 K$ by definition of hairs.}
\label{fig:prehairs}
\end{figure}
Let $\emptyset\neq \frh_K^\ell = \lbr z^\ell_1 , \dots, z^\ell_m\rbr $. By construction, any path 
$\eta_\ell :u_{\ell+1}\to v_\ell$  which is compatible with $\frh_\ell^K$ admits a decomposition (see Fig.~\ref{fig:prehairs})
\[
\eta_\ell = \eta_\ell^1 \cup \lambda_\ell^1\cup \eta_\ell^2\cup\lambda_\ell^2\cup \dots\cup \eta_\ell^m\cup \lambda_\ell^m\cup\eta_\ell^{m+1}, 
\]
where $\lambda_\ell^k : z_k^\ell\to w_k^\ell\in\Lpartial\UaK{K} (z_k^\ell)$  are precisely the paths which 
contribute to the hair $\frh_K^\ell$, whereas 
\[
\eta_\ell^1 :u_{\ell+1}\to z_1^\ell,\ \ \eta_\ell^2 :w_1^\ell\to z_2^\ell, \ \ \dots ,\ \ \eta_\ell^{m+1} :
w_m^\ell\to v_\ell 
\]
are the paths which might contribute to compatible pre-hairs.  Of course, some of the paths $\eta_\ell^i$ might be empty, and the last path $\eta_\ell^{m+1}$ might be too short.
However, the following requirements are enforced by construction (for the terminal path 
$\eta_\ell^{m+1}$ recall the 
double $\UaK{2K}$-ball condition in the definition of hairs):
\[
\max\lbr  \abs{R[ \eta_\ell^1]\cap\UaK{K}(z_1^\ell)}, \dots , \abs{R[ \eta_\ell^m]\cap\UaK{K}(z_m^\ell)}, 
\abs{R[ \eta_\ell^{m+1}]\cap\UaK{K} (w_m^\ell)}\rbr  \leq c_1 \alpha K .
\]
Consequently, we infer from~\eqref{eq:ThinReturns} that,
in the case of  non-empty $\frh_K^\ell  =\lbr z^\ell_1 , \dots, z^\ell_m\rbr $,
the contribution of all the compatible pre-hairs $\tilde\frh_K^\ell\sim \frh_K^\ell$ is
$\leqs \lb \log K\rb^{m+1}$, which is suppressed by the ${\rm e}^{-c_2 K m }$ cost of 
$ \frh_K^\ell$.  

Let us now consider the case $\frh_K^\ell = \emptyset$, with $\ell>0$. In that case, the path 
$\eta_\ell {: u_{\ell +1}\mapsto\Lpartial\UaK{K} (u_\ell )}$
satisfies by construction $\eta_\ell\cap \UaK{K} (u_{\ell -1} )= \emptyset$, and hence the 
corresponding contribution is bounded above as
$\leqs K^{d-1}\log K$. Indeed,  in terms of Green functions 
\[
\sumtwo{\eta: u_{\ell +1}\to  \Lpartial\UaK{K} (u_\ell )}{\eta\cap \UaK{K} (u_{\ell -1} )= \emptyset}
\sfp_d (\eta)
\leq\max_{z\in \Lpartial\UaK{K} (u_\ell  )} \sum_{y\in \Lpartial\UaK{K} (u_\ell  )}
G_{\bbZ^d\setminus \UaK{K} (u_{\ell -1} )} (z, y ) ,
\]
and the bound in question follows from \eqref{eq:GFLambda}.
However, for large enough $K$ the term $K^{d-1}\log K$  is suppressed by 
the ${\rm e}^{-K}$ cost of the $\ell$-th step of the trunk.  

The only remaining case is $\frh_K^0 = \emptyset$.  Then, necessarily,
\[
\eta_0\subseteq \hat \frD_0 \equiv \UaK{2K} (\frD\cup \{u_0\} ) \df \bigl(\frD\cup \{u_0\}\bigr)  + \UaK{2K}  .
\]
If $\UaK{2K} (u_0)\cap  \UaK{2K} (\frD ) = \emptyset$, then $\eta_0$ is confined to $\UaK{2K} (u_0)$ and, using again~\eqref{eq:GFLambda}, the corresponding contribution is bounded above by $\leqs\log K$, which is suppressed by the ${\rm e}^{-K}$ cost of the first step of the 
trunk. 

Thus, we are left to consider the situation when  $\UaK{2K} (u_0 )\cap  \UaK{2K} (\frD ) \neq \emptyset$. 
Using again the strong Markov property, we see that the contribution of $\eta_0$ is bounded above in terms of the SRW Green's function, see~\eqref{eq:GFLambda},
\[
 \sumtwo{\eta :u_1\to u_1}{\eta\subset\hat\frD_0} 
\sfp_d (\eta ) = G_{\hat\frD_0} (u_1 ,u_1 ) \leqs
\log d_1(u_1 ,\hat\frD_0^c ).
\]
There is nothing to worry about 
if $d_1(u_1 ,\hat\frD_0^c ) \leq c_5 K$.  If, however, $d_1(u_1 ,\hat\frD_0^c ) = rK > c_5 K$, 
then $\abs{\frd_K}\geq r^2$ and, consequently, 
$\log d_1(u_1 ,\hat\frD_0^c ) = \log (rK )$ is suppressed by the exponent 
$c_3 K\abs{\frd_K } \geq c_1 r^2 K$  in the interaction
 cost of the dirty set. 
\qed
\end{proof}
\subsection{Coarse-graining of decorated skeletons}
Recall the definition \eqref{eq:cone} of the forward cone $\calY_h$. 
The choice of $\nu$ in~\eqref{eq:cone} is dictated by Lemma~\ref{lem:nu}: 
For the rest of the paper we fix  $\nu = \nu (\kappa ) >0$ with 
$\kappa$ being  sufficiently small. 
\paragraph{Surcharge Inequality} As in \eqref{eq:cone}, 
given $h\in\partial\Ka$ and $u\in\bbR^d$, we define the surcharge
cost of $u$ as $\frs_h (u ) = \xi (u) -h\cdot u$. 

Accordingly, given a trunk $\frt_K = (u_0 ,u_1, \dots ,u_m )$ we define its surcharge cost as 
\begin{equation}
\label{eq:surcharge-tr}
\frs_h (\frt_K ) \df \sum_{\ell =0}^{m-1} \frs_h (u_{\ell +1 } - u_{\ell} ) .
\end{equation}
Since, by construction, $\abs{\xi (u_{\ell +1 } - u_{\ell}) - K}\leq {c_5}$ uniformly in 
scales $K$ and skeletons $\gamma_K$, we conclude that, for any fixed scale $K$, 
\be 
\label{eq:KsBound}
K\abs{\frt_K} - h\cdot x \geq \frs_h (\frt_K ) - c_5\abs{\frt_K} ,
\ee
uniformly in (sufficiently large) $x$ and  in  skeleton trunks $\frt_K$ compatible with paths $\gamma\in\calA_x$. 

Finally, the surcharge $\frs_h (\hat\gamma_K )$ of a decorated skeleton
 $\hat\gamma_K = \lbr \frt_K , \frd_K , \frh_K\rbr$ is defined as
\[
\frs_h (\hat\gamma_K ) \df  \frs_h (\frt_K ) + K\Bigl(  c_2 \sum_\ell \abs{\frh_K^\ell} +{c_4} \abs{\frd_K}\Bigr) .
\]
In view of~\eqref{eq:KsBound}, we can rewrite~\eqref{eq:skeleton} as
\be 
\label{eq:SurchSkeleton}
A_h (\hat\gamma_K |\;x ) \df {\rm e}^{h\cdot x}A ( \hat\gamma_K |\;x )
\leq {\rm exp}\Bigl\{ -\frs_h (\hat\gamma_K )\lb 1- \mathrm{o}_{\scriptscriptstyle K}(1)\rb  + \mathrm{o}_{\scriptscriptstyle K}(1)\, h\cdot x \Bigr\} .
\ee
\begin{proposition}
\label{prop:surcharge}
 For every $\delta >0$, there exists a scale $K_0 = K_0(\delta) $ such that, for every
$K\geq K_0$, 
\be 
\label{eq:surcharge}
A_h \bigl( \frs_h (\hat\gamma_K ) >2\delta \abs{x} \,\bigm|\,x \bigr) \leq \textrm{e}^{- \delta\abs{x}}, 
\ee
uniformly in $h\in\partial \Ka$ and all $\abs{x}$ sufficiently large.
\end{proposition}
\begin{proof}
 Let us fix a number $r >0$ sufficiently large. Two points $x$ and $y$ are said to be $K$-neighbors
if $\abs{x-y}\leq rK$. This imposes a $K$-connectivity structure on the vertices of $\bbZ^d$ and, consequently,
one can talk about $K$-connected sets. By construction, any decorated skeleton $\hat\gamma_K$ is 
$K$-connected on any scale $K$. Therefore, the number of $K$-connected  skeletons $\hat\gamma_K$ 
which contain the origin and have $N$ vertices, $\#\lb \hat\gamma_K \rb =N$, is bounded above 
by ${\exp(c_6 N\log K)}$. 
This is a particular case of Kesten's bound on the number of lattice animals 
(\cite{Kesten}, p.85). 

 Alternatively,  all vertices of $\hat\gamma_K$ can be visited by a
$K$-connected  path, starting at the origin
and having at most $2N$ steps.
Since the number of $K$-neighbors of a vertex is $\leqs (rK)^d$, the number of 
such different paths is $\leqs (rK)^{d2N} \leq {\exp(c_6 N\log K)}$.

This way or another, 
the stated bound readily follows since the labeling of each of these $N$ vertices as belonging to $\frt_K$, $\frd_K$ or $\frh_K$ only results in an additional factor $3^N$.

Since
\begin{equation}
\label{eq:surcharge-sk}
\frs_h (\hat\gamma_K) \geq K(\abs{\frt_K} + c_2\sum_\ell\abs{\frh_K^\ell} + c_4 \abs{\frd_K}) \geq \min(1,c_2,c_4)K\,\#\lb \hat\gamma_K \rb \equiv c_7 N K,
\end{equation}
the above entropy bound shows that we may ignore skeletons with $\#\lb \hat\gamma_K \rb \geq c_8 \abs{x}/K$. This leaves at most ${\exp(c_9 \abs{x}\log K/K)}$ skeletons to consider. 
Therefore~\eqref{eq:SurchSkeleton} implies~\eqref{eq:surcharge} as soon as $K$ is sufficiently
large.  
\qed
\end{proof}

\paragraph{Coarse-graining of decorated skeletons} Since $A(x)\asymp {\rm e}^{-\xi (x)}$, the bound
\eqref{eq:surcharge} is trivial whenever $\frs_h (x) >2\delta \abs{x}$.  
Given a decorated skeleton $\hat\gamma_K = \lbr\frt_K , \frd_K  , \frh_K\rbr$, 
let us say that $u\in\frt_K$ is an $h$-cone point of $\hat\gamma_K$ if 
$\hat\gamma_K\subset \lb u -\calY_h\rb\cup\lb u+\calY_h\rb$.  Define
\[
 \frt_{K}^* \df \setof{u\in\frt_{K, }}{\text{$u$ is not an $h$-cone point of $\frt_K$}},
\]
and define $\hat\gamma_{K}^* \df \{\frt_{K}^* , \frd_K  , \frh_K\}$.  We claim that the set
 $\hat\gamma_{K}^*$ is very sparse. Specifically, we partition of $\bbR^d$ into 
slabs $B_{K, j}^h \df \setof{x}{(j-1 )K \leq h\cdot x < jK}$. Given
a skeleton $\hat\gamma_{K}$, let us define 
\[
\calB^*_h  (\hat\gamma_{K, } ) \df \setof{j}{B_{K, j}^h\cap \hat\gamma_{K}^* \neq\emptyset} .
\]
Note that \eqref{eq:surcharge-sk} per se is not enough in order to control the size of $\abs{\calB^*_h }$.
Indeed we may use it in order to control the size of the trunk $\abs{\frt_{K}}$ and hence, by 
\eqref{eq:surcharge-tr}, to conclude that the fraction  of increments of $\frt_{K }$ with 
relatively high surcharges is small. However, being a cone point of a skeleton is a global property, 
and a-priori it might happen that even if $u\in\frt_{K}$ lies far away from the hairs and dirty boxes
of $\hat \gamma_{K }$, and even if adjacent increments of $\frt_{K}$ have small surcharges (and 
consequently lie inside $\calY_h$), the sleleton $\hat \gamma_{K }$ can still break away from
$\lb u -\calY_h\rb\cup\lb u+\calY_h\rb$ in some distant region. 
However, the surcharge price of breaking through at far away regions grows with the distance 
to these regions.  
This situation was treated in our early papers: 
By a  straightforward adaptation of the arguments developed in 
Section~3.3 of~\cite{IoffeVelenik-Annealed}, the surcharge inequality of
Proposition~\ref{prop:surcharge} implies
\begin{proposition}
 \label{prop:skeletons}
For every $\delta >0$, there exists a scale $K_1 = K_1(\delta)$ and an 
exponent $\epsilon >0$ such that, for every $K\geq K_1$, 
\be 
\label{eq:skeletons}
A_h \bigl( \abs{\calB^*_h  (\hat\gamma_K ) }  >  \delta \frac{\abs{x}}{K} \bigm| x \bigr) \leq \textrm{e}^{- \epsilon \abs{x}}, 
\ee
uniformly in $h\in\partial \Ka$ and all $\abs{x}$ sufficiently large. 
\end{proposition}
Since $\xi$ is a norm on $\bbR^d$, we may assume that $\nu\xi(x ) >\epsilon \abs{x}$ for all $x\neq 0$. In this case, \eqref{eq:skeletons} trivially holds for $x$-s lying outside the cone
$\calY_h$ (remember that $A_h(x) = e^{h\cdot x} A(x) \asymp {\rm e}^{-\frs_h (x)}$).

\subsection{Coarse-graining of microscopic polymers} 
\label{sub:CG-micro}
In order to define the irreducible splitting of microscopic paths $\gamma$,
we need an appropriate enlargement of the forward cone $\calY_h$ which was defined in 
 \eqref{eq:cone}. 
Recall that $\nu = \nu (\kappa )$ is fixed and it corresponds to a sufficiently small
choice of $\kappa$ through Lemma~\ref{lem:nu}. 
Let $0 <\nu < \tilde\nu <1$. Define
\be  
\label{eq:tildecone}
\tilde{\calY}_h \df \setof{x}{h\cdot x \geq (1-\tilde\nu )\xi (x)}.
\ee
There is a strict inclusion 
$\calY_h \subset \tilde{\calY}_h$. 
In the sequel we choose and fix  
$\tilde\nu$ in such a way that:
\smallskip

(i) 
 The interior of $\tilde\calY_h$ contains a lattice direction
$\sfe_h$, but, still, the aperture of $\tilde{\calY}_h$ is less than $\pi$. 

\begin{figure}[t]
\begin{center}
\scalebox{0.6}{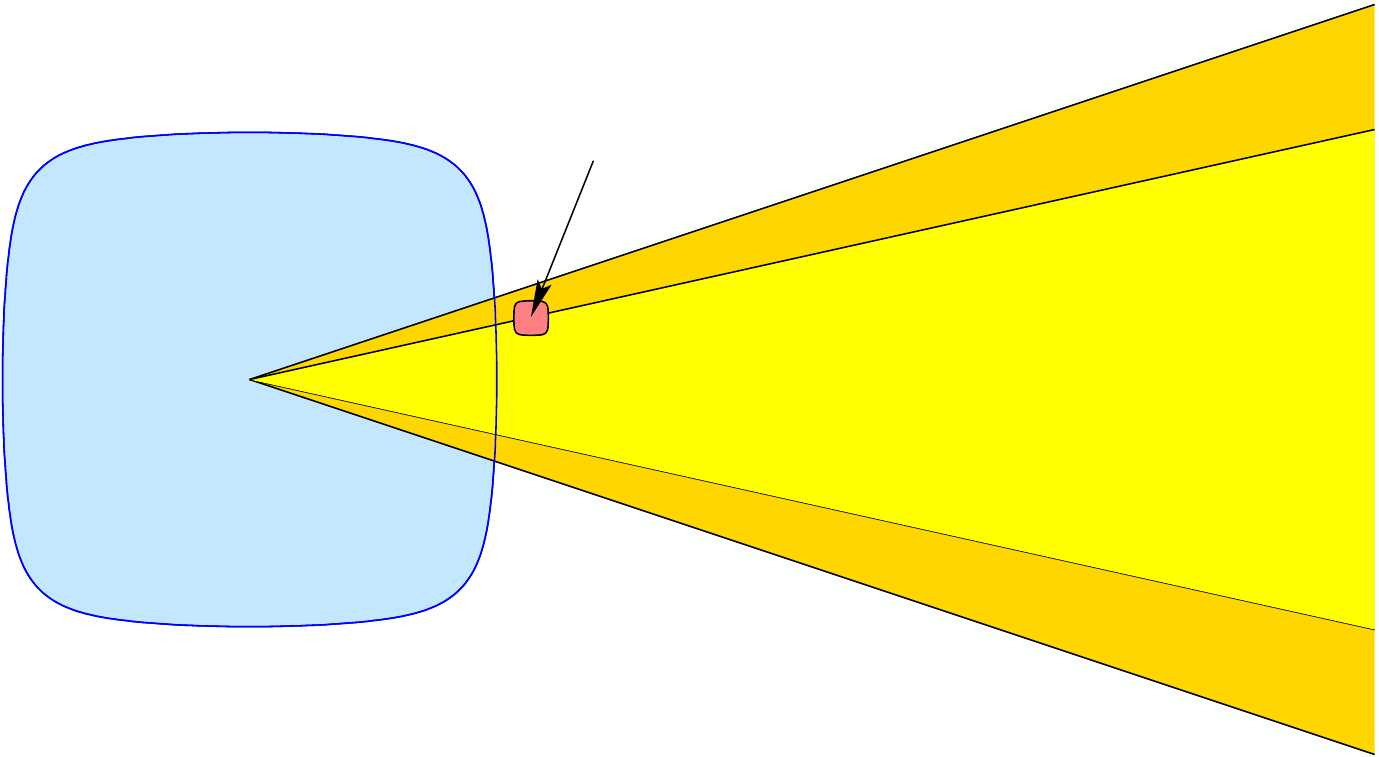}
\end{center}
\caption{The cones $\calY_h$ and $\tilde{\calY}_h$ are well-separated: any translate of $\UaK{\bar{c}\kappa}$ centered at a point of $\calY_h\setminus\Ua$ is contained inside $\tilde{\calY}_h$.}
\label{fig:cones}
\end{figure}
(ii) Cones $\tilde\calY_h$ and $\calY_h$ are well-separated outside 
$\Ua$ in the following sense (see Fig.~\ref{fig:cones}): Let $(1-\nu )(1+\kappa )> 1$ 
(see Proof of Lemma~2). Then,
with $\bar{c}\df\max_{x\in\Ua}\abs{x}$,
\be  
\label{eq:Ycone-sep}
\UaK{\bar{c}\kappa}(x) \subset \tilde\calY_h\quad\text{for all $x\in\calY_h\setminus\Ua$.}
\ee

\paragraph{Definition} Let us say
that $u\in\gamma
$ is an $h$-cone point of $\gamma$ if 
the local time $\ell_\gamma [ u ]=1$ and $\gamma \subset (u- \tilde{\calY}_h) \cup 
(u +\tilde{\calY}_h)$.
\begin{theorem}
 \label{thm:cone-points}
There exists  $\chi  >0$ such that 
\be  
\label{eq:cone-points}
A_h (\text{$\gamma$ has less than two  $h$-cone points}\,|\,  x ) \leq \textrm{e}^{-\chi \abs{x}} ,
\ee
uniformly in $h\in\partial \Ka$ and $x$ large enough.
\end{theorem}
\begin{definition}
 \label{def:irreducible}
Let $h\in\partial\Ka$.

\noindent
$\vartriangleright$ The set of backward irreducible paths $\calT_{\sfb}$ contains those paths $\gamma = (u_0 , \dots , u_n )$ such that $\ell_\gamma (u_n ) = 1$, $\gamma\subset u_n - \tilde\calY_h$
and $\gamma$ does not have $h$-cone points other than $u_n$.

\noindent
$\vartriangleright$ The set of forward irreducible paths $\calT_{\sff}$ contains those paths $\gamma = (u_0 , \dots , u_n )$ such that $\ell_\gamma (u_0 ) = 1$, $\gamma\subset u_0 + \tilde\calY_h$ and $\gamma$ does not have $h$-cone points other than $u_0$.

\noindent
$\vartriangleright$ The set of irreducible paths $\calT$ contains the paths such that $\ell_\gamma (u_0 ) = \ell_\gamma (u_n ) = 1$, $\gamma\subset (u_0 + \tilde\calY_h)\cap(u_n - \tilde\calY_h)$, with no $h$-cone points besides $u_0$ and $u_n$.
\end{definition}
\begin{remark}
In the sequel, we shall frequently identify irreducible paths with their appropriately
chosen spatial shifts. Furthermore, we shall use notation $\gamma_1\amalg\gamma_2\amalg\dots $
(instead of $\gamma_1\cup\gamma_\cup \cdots $) whenever we are talking about 
concatenation of \em{irreducible} paths.
\end{remark}
The following central result, 
which sets up the stage for
the subsequent analysis of polymer measures at critical drifts, 
 is an immediate consequence of Theorem~\ref{thm:cone-points}.
\begin{theorem}
\label{thm:renewal}
Let  $h\in\partial \Ka$. For any $x\in \bbZ^d$, either $A_h (x) \leq {\rm e}^{-\chi\abs{x}/2}$ or 
\be 
\label{eq:renewal}
\begin{split}
A_h (x)&\lb 1 +
\smo{{\rm e}^{-\frac{\chi \abs{x}}{2}}}
\rb\\
&\quad = 
\sumtwo{\gamma_\sff\in\calT_\sff}{\gamma_\sfb\in\calT_\sfb}
\sum_{N=1}^\infty  \sum_{\gamma_1, \dots , \gamma_N\in \calT}
a_h (\gamma_\sfb )a_h (\gamma_\sff )\prod_1^N a_h (\gamma_i )
\1_{\lbr \gamma_\sfb\amalg\gamma_1\amalg\dots\amalg\gamma_N\amalg\gamma_\sff
\in\calA_x\rbr}.
\end{split}
\ee
\end{theorem}
\begin{proof}[of Theorem~\ref{thm:cone-points}]
Given a decorated skeleton $\hat\gamma_K$, let us define
\[
\bbB^*_h  (\hat\gamma_K ) \df \bigcup_{j \in \calB^*_h  (\hat\gamma_K )} B_{K, j}^h \quad 
{\rm and}\quad \bbB_h  (\hat\gamma_K ) \df \bbR^d\setminus \bbB^*_h  (\hat\gamma_K ) .
\]
Both $\bbB^*_h$ and $\bbB_h$ are disjoint unions of strips orthogonal to the $h$-direction. 
By construction, the only vertices of $\hat\gamma_K$ which lie inside $\bbB_h$ are the 
cone points of the trunk $\frt_K = \lb u_0 , \dots , u_m \rb $. Let  $r\in\bbN$. A string of consecutive 
vertices 
$\frt_{K, j}\df \lb u_{j-2r}, \dots , u_{j+2r} \rb  \subset \frt_K$ is called a \emph{regular $r$-stretch} if
all these vertices belong to the same connected component of $\bbB_h$.  Lattice points which lie close
to the centers 
$u_j$ of regular stretches are candidates for $h$-cone points of paths.  Of course, one should 
make appropriate choices of $r$ and the coarse-graining scale $K$. Let us explain how these choices are made:
\smallskip 

\noindent
$\blacktriangleright\;$\textsc{Step} 1: Choice of scales.\newline
\emph{Choice of $r$.}  By construction, there exists $\psi <\infty$ such that, on all scales $K$, the paths
$\gamma\subset \cup_{y\in\hat\gamma_K} B_{{\psi} K}(y)$ whenever $\hat\gamma_K$ is the 
decorated skeleton of $\gamma$. We choose $r$ in such a way that, for all large enough scales $K$,
\be 
\label{eq:rchocie}
\bigcup_{y\in\hat\gamma_K\setminus\lbr u_{j-r }, \dots , u_{j+r}\rbr}B_{{\psi} K}(y)
\subset \bigl( u_j - \tilde\calY_h\bigr) \cup\bigl( u_j +\tilde\calY_h\bigr) ,
\ee
provided that $\frt_{K, j}$ is a regular $r$-stretch.  This is possible by definition of 
cone points (of skeletons) and in view of the strict inclusion $\calY_h\subset\tilde\calY_h$.  

\noindent
\emph{Choice of $K$.}
Let $r$ be  chosen as above. Regular stretches $\frt_{K,j_1}, \ldots, \frt_{K, j_m}$ are called 
disjoint if $j_{\ell}+2r \leq j_{\ell +1} -2r$ for all $\ell = 1, \dots, m-1$. We rely on the following consequence
of Proposition~\ref{prop:skeletons}: There exist $\delta_1 >0$, $\epsilon_1 >0 $ and a scale 
$K$, such that, 
\be   
\label{eq:Kchoice}
A_h \bigl( \text{$\hat\gamma_K$ has less than $\frac{\delta_1\abs{x}}{K}$ 
disjoint regular $r$-stretches} \bigm| x \bigr) \leq \textrm{e}^{- \epsilon_1 \abs{x}}, 
\ee
uniformly in $h\in\partial\Ka$ and $\abs{x}$ sufficiently large.
\smallskip 

\noindent
$\blacktriangleright\;$\textsc{Step} 2: Domination by independent Bernoulli random variables.\newline
We fix $r$ and $K$ as above and proceed with the proof. Let $\hat\gamma_K$ be a decorated
skeleton which contains $N\geq \delta_1\abs{x}/K$ disjoint regular stretches 
$\frt_{K, j_1}, \dots ,\frt_{K, j_N}$. Then, for any $\gamma\sim\hat\gamma_K$, the  decomposition~\eqref{gleta} gives rise to
\be 
\label{eq:stretches} 
\gamma = \rho_0\amalg\lambda_1\amalg\rho_1\amalg\dots \amalg\lambda_N\amalg\rho_{N} ,
\ee
where $\lambda_\ell$ is the portion of $\gamma$ from $u_{j_\ell - r}$ to $u_{j_\ell +r}$.  Explicitly, 
\be 
\label{eq:lambdaStretch}
\lambda_\ell \df \gamma_{j_\ell -r}\amalg\eta_{j_\ell -r}\amalg\dots
\amalg\gamma_{j_\ell +r-1}
\amalg\eta_{j_\ell +r-1}. 
\ee
Now, things are set up in such a way that if $\gamma_{j_\ell}$ contains an $h$-cone point 
of $\lambda_\ell$, then 
it is automatically  an $h$-cone point of the whole path $\gamma$.
This follows from~\eqref{eq:rchocie} and the observation
that $\gamma\setminus\lambda_\ell$ is confined to the union therein. 
Furthermore, since we are talking about regular stretches,  the pre-hairs $ \frh_K^i =\emptyset$
for any $\ell = 1, \dots, N$ and any $i = j_\ell -r , \dots, j_{\ell}+r$.  Which means that $\lambda_1, \dots ,
\lambda_N$ are disjoint. 
Define the event
\[
 \calE (\delta ,x ) \df \Bigl\{\text{$\gamma$ has less than $\delta\frac{\abs{x}}{K}$ $h$-cone points}\Bigr\} .
\]
In the sequel, we shall use the following notation for {\em restricted} partition functions: 
\[
A_h \bigl( \cdot \bigm| \calC\bigr) \df \sum_{\gamma\sim\calC} a_h(\gamma)
\1_{\lbr \gamma\in \cdot\rbr}
,
\]
where the sum is over all paths compatible with some constraints $\calC$ (e.g., a given decorated skeleton); the constraint $\calC$ will always imply that the endpoint is $x$.
We claim that there exist $\delta_2 = \delta_2 (K) >0$ and $\epsilon_2 =\epsilon_2 (K)>0$ such 
that
\be  
\label{eq:EBound}
A_h \bigl( \calE (\delta_2 ,x ) \bigm| \hat\gamma_K\bigr) \leq \mathrm{e}^{-\epsilon_2\abs{x}} ,
\ee
uniformly in decorated skeletons $\hat\gamma_K$ which have $N\geq \delta_1\abs{x}/K$ disjoint stretches. 
Indeed, in the notation  of~\eqref{eq:stretches}, for  any such skeleton,
\[
 A_h \lb \calE (\delta_2 ,x ) ~\big|\; \hat\gamma_K\rb \leq \max_{\underline{\rho}\df (\rho_0, \dots , \rho_N)}
A_h \lb \calE (\delta_2 ,x ) ~\big|\; \hat\gamma_K ;\underline{\rho} \rb .
\]
This is controlled by comparison with independent Bernoulli random variables. Define
\be  
\label{eq:p-rho}
p = p (\underline{\rho}) \df \min_\ell \frac{\sum_{\lambda_\ell\sim\frt_{K, j_\ell}} a(\lambda_\ell \big|
\underline{\rho})\1_{\lbr\text{$\gamma_{j_\ell} $ contains an $h$-cone point of $\lambda_\ell$ }\rbr} }{
\sum_{\lambda_\ell\sim\frt_{K, j_\ell}} a(\lambda_\ell \big|
\underline{\rho})} .
\ee
Above, $\lambda_\ell\sim\frt_{K, j_\ell}$ is just a short-hand notation for~\eqref{eq:lambdaStretch},
and 
\[
a(\lambda_\ell \big|
\underline{\rho}) \df {\rm exp}\lbr -\Phi (\lambda_\ell \big|
\underline{\rho}) - \abs{\lambda_\ell}\log (2d)\rbr
\]
 with 
\[
 \Phi (\lambda_\ell \big|
\underline{\rho}) \df \sum_x
\lbr 
\phi(\ell_{\underline{\rho }}[x] +\ell_{\lambda_\ell} [x])
-\phi(\ell_{\underline{\rho }}[x]) \rbr.
\]
Since $\lambda_1, \dots ,\lambda_N$ are disjoint
\[
 A_h \bigl( \calE (\delta_2 ,x ) \bigm| \hat\gamma_K  ;\underline{\rho}  \bigr) \leq \bbP_p\bigl( 
X_1 +\dots +X_N <  \delta_2 \frac{\abs{x}}{K}\bigr) ,
\]
where $X_1, \dots ,X_N$ are i.i.d. Bernoulli($p$) random variables. 

It remains to derive a 
 {strictly positive
lower bound on $p (\underline{\rho} )$, which would hold uniformly 
 in $\hat\gamma_K$ and $\underline{\rho}$.
Note that we are not shooting for optimal estimates on $p$ in~\eqref{eq:p-rho}.  Any uniform estimate
which depends only on $K$ and $r$ would do.
}
\smallskip

\noindent
$\blacktriangleright\;$\textsc{Step} 3: Upper bound on denominator.\newline
Let $\frt_{K ,j_\ell}$ be a regular stretch of a skeleton $\hat\gamma_{K}$. We continue to rely on the decomposition~\eqref{eq:stretches} of compatible paths $\gamma\sim\hat\gamma_K$.  Since the interacting
potential $\phi$ is 
non-decreasing, 
$a (\lambda_\ell |\underline{\rho }) \leq  \lb 1/2d\rb^{\abs{\lambda_\ell }}
$.
Moreover, we are talking about regular stretches, so that
compatible paths $\lambda_\ell  : u_{j_\ell -r}\to  u_{j_\ell +r}$ are necessarily confined:
$\lambda_\ell \subseteq \UaK{2K} \lb \frt_{K ,j_\ell }\rb$. 
It follows that the denominator in~\eqref{eq:p-rho} is bounded above as
\be  
\label{eq:denom-p}
\sum_{\lambda_\ell\sim\frt_{K, j_\ell}} a(\lambda_\ell \big|
\underline{\rho}) 
\leqs  \log K.
\ee
\smallskip

\noindent
$\blacktriangleright\;$\textsc{Step} 4: Surgeries and lower bound on numerator.\newline
In order to derive a lower bound on the numerator it would 
suffice  to show that there exists $c_{10}<\infty$ such that, for any $\underline{\rho}$
in~\eqref{eq:stretches}, one is able to produce a path
\be 
\label{eq:lambdaStretch-star}
\lambda_\ell^* = \gamma^*_{j_\ell -r}\amalg\gamma^*_{j_\ell -r+1}\amalg\dots \amalg \gamma^*_{j_\ell +r -1} , 
\ee
which is compatible with~\eqref{eq:lambdaStretch}, $\lambda_\ell^*\sim\frt_{K, j_\ell}$,
contains an $h$-cone point
in its segment 
$\gamma_{j_\ell}^* $, 
and satisfies $a (\lambda_\ell ^*|\underline{\rho }) \geq {\rm e}^{-c_{10} K}$.  Indeed, if such 
$\lambda_\ell^*$ exists,  then, 
in view of~\eqref{eq:denom-p}, a substitution into~\eqref{eq:p-rho} gives the following target uniform
lower bound on $p$:
\be 
\label{eq:BoundOnP}
\min_{\underline{\rho}} p (\underline{\rho})  \geqs \frac{ {\rm e}^{- c_{10} K}}{\log K}  .
\ee
We claim that the required properties of $\lambda_\ell^*$ in its decomposition~\eqref{eq:lambdaStretch-star}
are secured by the following set of conditions:

(a) For all $k=-r, \dots, r-1$, the path $\gamma^*_{j_\ell +k}: u_{j_\ell +k}\mapsto u_{j_\ell +k+1}$
is self-avoiding and, apart from its end-point $u_{j_\ell +k+1}$, lies inside  
$
\UaK{K}(u_{j_\ell +k})\setminus  
\cup_{m <k}\UaK{K} (u_{j_\ell +m })$. 

(b)  There exists a geometric constant $g_1= g_1(d)$ (independent of $K$) such that all the paths 
satisfy $\abs{\gamma^*_{j_\ell +k}}\leq g_1K$. 

(c)  There exists a geometric constant $g_2= g_2 (d) \leq \bar{c}$ (see~\eqref{eq:Ycone-sep}) such that, for all $k= -r, \dots, r-1$,
the distance ${\rm d}\bigl(  \gamma^*_{j_\ell +k} , [u_{j_\ell +k},  u_{j_\ell +k+1}]\bigr) \leq g_2\kappa K$, where
$\kappa$ is chosen to be sufficiently small according to Lemma~\ref{lem:nu}. 

(d) $\gamma^*_{j_{\ell} }$ has an $h$-cone point at the lattice approximation $u^*_{j_{\ell}}$
of the mid-point of  the segment $[u_{j_{\ell} }, u_{j_{\ell} +1}]$. 
\smallskip

\noindent
Indeed, property (a) simply means that $\lambda_\ell^*$ is compatible with the decomposition 
\eqref{eq:lambdaStretch}. 
Property (b) implies that
\[
 \sum_x
\lbr 
\phi\lb \ell_{\underline{\rho }}[x] +\ell_{\lambda_\ell^*} [x]\rb
- 
\phi\lb \ell_{\underline{\rho }}[x]\rb \rbr
 \leq 
\sum_x
\1_{\lbr \ell_{\lambda_\ell^*}[x] >0\rbr}
\phi\lb \ell_{\lambda_\ell^*} [x]\rb \leq 
 2 rg_1K 
\phi (1) .
\]
It follows that $a (\lambda^*_\ell |\underline{\rho }) \geq {\rm e}^{- 2rg_1 K(\phi(1)+\log(2d))}\df {\rm e}^{- c_{10} K}$. 

Finally, since all the increments $u_{j_\ell +k +1} - u_{j_\ell +k}$  belong to the 
narrow cone $\calY_h$, by~\eqref{eq:Ycone-sep} properties (c) and (d) imply that $u^*_{j_{\ell}}$
is an $h$-cone point of the whole path $\lambda^*_\ell$.
\smallskip 

Therefore, it only remains to check that there exist paths $\lambda_\ell^*$ which enjoy (a)-(d).  Let
us describe how $\gamma^*_{j_\ell +k}$ in~\eqref{eq:lambdaStretch-star} could be constructed
(see Figure~\ref{fig:path}).

\begin{figure}[t]
\begin{center}
\scalebox{.5}{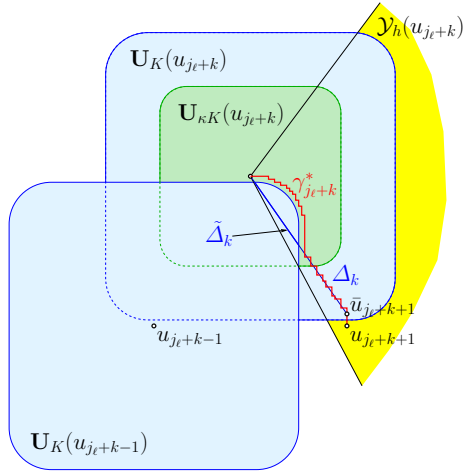}
\end{center}
\caption{Construction of $\gamma^*_{j_\ell +k}$.}
\label{fig:path}
\end{figure}

Lemma~\ref{lem:nu} and lattice symmetries will play a role.  First of all, by construction of the skeletons, 
$u_{j_\ell +k+1}$ has a lattice neighbor $\bar u_{j_\ell +k+1}\in \UaK{K} ( u_{j_\ell +k})$. The edge
$\lbr \bar u_{j_\ell +k+1}, u_{j_\ell +k+1}\rbr$ will be the last step of $\gamma^*_{j_\ell +k}$.  Consider
the segment $\Delta_k\df [u_{j_\ell +k}, \bar u_{j_\ell +k+1}]$. By convexity, $\Delta_k\subset \UaK{K} (u_{j_\ell +k} )$. We claim that there is a self-avoiding lattice approximation
$\gamma_{j_\ell +k}$ of $\Delta_k$ which runs from $u_{j_\ell +k} $ to $\bar u_{j_\ell +k+1}$, satisfies
property (b) and stays inside $\UaK{K} (u_{j_\ell +k} )$.  This follows from convexity and lattice symmetries of the model (and hence of $\UaK{K}$).  By Lemma~\ref{lem:nu}, $\Delta_k$, and hence $\gamma_{j_\ell +k}$,
is well separated from $\cup_{m<k-1}\UaK{u_{j_\ell +m}}$. However, since at this stage we cannot
rule out facets on $\partial\Ua$,  it might happen that 
$\tilde\Delta_k\df \Delta_k\cap \UaK{K}(u_{j_\ell +k-1}) \neq \emptyset$ (see Fig.~\ref{fig:path}).  Thus, $\gamma_{j_\ell +k}$ might violate (a) and, thereby, be 
incompatible with the skeleton construction. In such a case, it should be modified. However, 
by the very same Lemma~\ref{lem:nu}, $\tilde\Delta_k \subset \UaK{\kappa K}(u_{j_\ell +k})$.  It remains
to take any self-avoiding path of minimal length which connects $u_{j_\ell +k}$ to $\gamma_{j_\ell +k}$ outside
$\UaK{K}(u_{j_\ell +k-1})$.  

An additional care is needed for securing property (d) for the path $\gamma^*_{j_\ell}$. To this end, we construct
the lattice approximation $\gamma_{j_\ell}$ as above with an additional requirement that $u^*_{j_\ell}$
 belongs to $\gamma_{j_\ell}$ and is an $h$-cone point of the latter. This is evidently possible. By then, 
\eqref{eq:Ycone-sep} implies that $u^*_{j_\ell}$ is also an $h$-cone point of the modified 
path $\gamma_{j_\ell}^*$.
\qed
\end{proof}
\section{Proofs of the Main Results}
\label{sec:Proofs}
Theorem~\ref{thm:renewal} paves the way for a description of critical 
$d$-dimensional self-attractive random walks in terms of effective $(d-1)$-dimensional  
directed walks with random time steps.
\subsection{Effective Random Walk}
\label{sub:RW}
The logarithmic asymptotics $A (x)\asymp {\rm e}^{-\xi (x)}$ and \eqref{eq:support} imply that the set $\Ka$ is the domain of convergence of $\sum_x A_h (x)$. More precisely, 
\be  
\label{eq:Kshape}
 h\in {\rm int}\lb \Ka\rb\Rightarrow \sum_x A_h (x) <\infty\quad {\rm and}
\quad 
h\not\in\Ka \Rightarrow \sum_x A_h (x) = \infty .
\ee
Together with Theorem~\ref{thm:renewal}, this leads to the following set of properties which 
characterize critical drifts $h\in\partial\Ka$.
\begin{theorem}
 \label{thm:critical-drifts}
Let $h\in\partial\Ka$ and let the cone $\tilde\calY_h$ and the corresponding set of irreducible 
paths $\calT$ be the same as in~\eqref{eq:renewal}.  Then $\bbP_h (\cdot )\df a_h (\cdot )$ is a probability
distribution on $\calT$, that is
\[
 \sumtwo{\gamma\in\calT}{\gamma (0) = 0} a_h (\gamma ) =1 .
\]
Furthermore, the random variables $\sfT = \sfT (\gamma ) = \abs{\gamma }$ (length) and 
$\sfX = \sfX (\gamma ) = \gamma (\sfT )-\gamma (0)$ (displacement) satisfy:

(a)  The support ${\rm supp}(\sfX)\subseteq \tilde \calY_h$ and 
the distribution of $\sfX$  has exponential tails: There exists $\kappa >0$ such that, uniformly 
in $x$, 
\be 
\label{eq:Xdecay}
 \bbP_h \lb \abs{\sfX } \geq x\rb \leqs {\rm e}^{-\kappa x} .
\ee

(b) The distribution of $\sfT$ has a stretched exponential decay: There exists $\kappa_T>0$ such that
uniformly in $t$,
\be 
\label{eq:stretched}
\bbP\lb \sfT \geq t\rb \leqs {\rm e}^{-\kappa_T {t}^{1/3} } .
\ee
\end{theorem}
Exponential decay of $\bbP_h \lb \abs{\sfX } \geq x\rb$ is the content of~\eqref{eq:cone-points}.
In order to check~\eqref{eq:stretched}, we first define, given $v-u\in\tilde\calY_h$, the diamond
shape
\be 
\label{eq:diamond}
\sfD_h (u; v ) \df (u+\tilde\calY_h) \cap (v- \tilde\calY_h) .
\ee
Then,
\[
 \bbP_h\lb
\sfT\geq t\rb
\leq \bbP_h(\abs{X} \geq \epsilon t^{1/3}) + \bbP_h(\abs{X} < \epsilon t^{1/3}, T\geq t),
\]
and, using~\eqref{eq:ProbStayLongInSmallSet}, the last term is bounded above by
\[
\sum_{\abs{x}\leq \epsilon t^{1/3}} e^{h\cdot x} \sumtwo{\gamma\in\calA_x}{\gamma\subseteq \sfD_h (0; x )}
\sfp_d (\gamma ) \1_{\lbr \abs{\gamma}\geq t\rbr}
\leqs {\rm e}^{\epsilon \abs{h} t^{1/3}} (\epsilon t^{1/3})^d \textrm{e}^{-\kappa^\prime t/(\epsilon t^{1/3})^{2}},
\]
for some geometric  {constant} $\kappa^\prime = \kappa^\prime (d)>0$. 
\begin{remark}
 \label{rem:remBoundary}
Note that a straightforward modification of the above arguments implies exponential and, respectively, 
stretched exponential decay of weights on backward and forward irreducible paths $\gamma_\sfb$ 
and $\gamma_\sff$: For $*=\sfb ,\sff$, 
\be 
\label{eq:remBoundary}
\sum_{\gamma_*\in\calT_*} a_h (\gamma_* )\1_{\lbr \abs{\sfX (\gamma_* )}\geq x\rbr}
\leqs {\rm e}^{-\kappa x}\quad {\rm and}\quad 
\sum_{\gamma_*\in\calT_*} a_h (\gamma_* )\1_{\lbr {\sfT (\gamma_* )}\geq t\rbr}
\leqs {\rm e}^{-\kappa_T t^{1/3}}. 
\ee
\end{remark}

\subsection{Geometry of $\Ka$ and $\Ua$}
\label{sub:geometry}
Once 
~\eqref{eq:Kshape} and~\eqref{eq:Xdecay} are established, the geometry of $\partial \Ka$ 
can be studied exactly as in the super-critical case  
\cite{IoffeVelenik-Annealed}. Namely, 
let $h\in\partial\Ka$. 
Then the shape of $\partial\Ka$
in a neighborhood of $h$ can be described as follows
(see Lemma~4.1  in \cite{IoffeVelenik-Annealed}): There exists $\delta = \delta (h) >0$ such that,
for any $f\in\bbR^d$ with $\abs{f} <\delta$, 
\be 
\label{eq:pKa}
h+f \in \partial\Ka ~ \Leftrightarrow ~  \bbE_h {\rm e}^{f\cdot \sfX (\gamma )} =1 .
\ee
Since $\tilde\calY_h$ is a proper cone with a non-empty interior, the distribution 
of $\sfX $ under $\bbP_h$ is genuinely $d$-dimensional, in particular the covariance matrix
of $X$ is non-degenerate.
Hence, the analytic implicit function theorem (see e.g.~\cite{Kaup}) applies.
The claim of Theorem~\ref{thm:B} is an immediate consequence.
By duality, the boundary $\partial\Ua$ is also locally analytic and has a uniformly strictly 
positive Gaussian curvature.

\subsection{Partition function $A_h^n$}
Let us fix $\epsilon\in (0,1/3 )$. We decompose $A_h^n$ as 
\be  
\label{eq:Ahn}
A_h^n = \sum_{x\in \UaK{n^\epsilon}} A_h^n (x) + \sum_{x\not\in\UaK{n^\epsilon}} A_h^n (x) .
\ee
By Theorem~\ref{thm:renewal} the second term is, up to a correction of order $
{\rm e}^{-O(\chi n^\epsilon/2)}$, 
\be 
\label{eq:AhnMain}
 \sumtwo{\gamma_\sff\in\calT_\sff}{\gamma_\sfb\in\calT_\sfb}
a_h (\gamma_\sfb )a_h (\gamma_\sff )
\sum_{N=1}^\infty \bbP^\otimes_{h}\lb T_1+\cdots +T_N = 
n- T(\gamma_\sfb) - T (\gamma_\sff )\rb .
\ee
By the usual renewal theorem, 
\[
 \lim_{m\to\infty}
\sum_{N=1}^\infty \bbP^\otimes_{h}\lb T_1+\cdots +T_N = m\rb = \frac{1}{\bbE_h T} .
\]
On the other hand, in view of the stretched exponential decay of the irreducible weights 
~\eqref{eq:remBoundary} (and by lattice symmetries if we choose $\tilde\calY_{-h} = -\tilde\calY_h$,
which we shall do),
\be 
\label{eq:bondary-weights}
\sum_{t=1}^\infty \sum_{T (\gamma_\sfb )=t} a_h (\gamma_\sfb ) = 
\sum_{t=1}^\infty \sum_{T (\gamma_\sff )=t}a_h (\gamma_\sff )  \df C_h <\infty .
\ee
It follows that the second term in~\eqref{eq:Ahn} satisfies:
\be 
\label{eq:second}
\lim_{n\to\infty} \sum_{x\not\in\UaK{n^\epsilon}} A_h^n (x) = \frac{C_h^2 }{\bbE_h T} .
\ee
It remains to show that the first term in~\eqref{eq:Ahn} is negligible:
\label{sub:An}
\begin{lemma}
\label{lem:small-x}
Fix $\epsilon \in (0, 1/3)$. Then, for all $n$ sufficiently large, 
\begin{equation}
 \label{eq:small-x}
A^n_h (x) \leqs {\rm e}^{-c_{11} n^\epsilon} ,
\end{equation}
uniformly in $h\in\partial\Ka$ and $x\in \UaK{n^\epsilon}$ .
\end{lemma}
\begin{proof}[of Lemma~\ref{lem:small-x}]
By definition, $A_h^n (x) = {\rm e}^{h\cdot x}A^n (x)$. 
Let us pick a number $R$ large enough and decompose
\be 
\label{eq:two-sums}
A^n (x) = \sumtwo{\gamma :0\mapsto x}{\gamma\subseteq \UaK{Rn^\epsilon}}
\sfp_d (\gamma ) {\rm e}^{-\Phi (\gamma )}\1_{\lbr\abs{\gamma} = n\rbr} +
\sumtwo{\gamma :0\mapsto x}{\gamma\not\subseteq \UaK{Rn^\epsilon}}
\sfp_d (\gamma )
{\rm e}^{-\Phi (\gamma )}\1_{\lbr\abs{\gamma} = n\rbr} .
\ee
The first sum in~\eqref{eq:two-sums} is bounded above by ${\rm e}^{-c_{12}n^{1-2\epsilon}}$ using~\eqref{eq:GFLambda}.
Since  $\abs{h\cdot x}\leqs n^\epsilon$ uniformly in $x\in\UaK{n^\epsilon}$, and since $\epsilon <1-2\epsilon$,
 this indeed complies with the right-hand side of~\eqref{eq:small-x}.

As for the second term define $\tau = \tau (\gamma ) \df \min\setof{\ell}{\xi (\gamma (\ell ) )= \max _m
\xi (\gamma (m ))}$, and, accordingly, define $z = z (\gamma ) \df \gamma (\tau )$. We can then bound the
second term in~\eqref{eq:two-sums} by
\[
 \sum_{z\,:\,\xi (z ) \geq Rn^\epsilon} A \lb z (\gamma ) = z |x \rb .
\]
However, $A\lb  z (\gamma ) = z | x\rb \leq A (z) G_{\UaK{\xi (z)}} (x ,x)\leqs {\rm e}^{-\xi (z )}\log\xi (z)$, using~\eqref{eq:GFLambda} once more.
Since $\abs{h\cdot x}\leqs n^\epsilon$ and we are summing with respect to $z$ with $\xi (z)\geq Rn^\epsilon$,
the second term in~\eqref{eq:two-sums} also complies with the right-hand side of
\eqref{eq:small-x}. 
\qed
\end{proof}

Altogether we have proved:
\begin{theorem}
 The asymptotics of the critical partition function is given by
\be 
\label{eq:criticalAhn}
\lim_{n\to\infty} A_h^n = \frac{C_h^2 }{\bbE_h T},
\ee
uniformly in $h\in\partial\Ka$. 
\end{theorem}

\subsection{Existence of the limiting velocity} 
It should be obvious from the above discussion that 
for any critical $h\in\partial\Ka$, 
the natural candidate for the 
limiting velocity is 
\be  
\label{eq:l-velocity}
v (h) \df \lim_{n\to\infty} \bbA_h^n\bigl(\frac X n\bigr) = \frac{\bbE_h \sfX}{\bbE_h \sfT}.
\ee
In view of~\eqref{eq:remBoundary}, \eqref{eq:bondary-weights} and~\eqref{eq:criticalAhn} we may, at this level of resolution,
ignore the boundary pieces $\gamma_\sfb$ and $\gamma_\sff$ in the decomposition~\eqref{eq:renewal}.
Therefore, we need to check that
\be 
\label{eq:EfVelocity}
\bbE_h \sfT\lim_{n\to\infty} \sum_{N\leq n} \bbE_h^\otimes  \Bigl( \frac1{n} \sum_1^N \sfX_i ~; ~\sum_1^N \sfT_i = n\Bigr)
= \frac{\bbE_h \sfX}{\bbE_h \sfT}. 
\ee
By exchangeability, 
\[
 \bbE_h^\otimes  \Bigl( \frac1{n} \sum_1^N \sfX_i ~; ~\sum_1^N \sfT_i = n\Bigr) = 
\frac{N}{n} \sum_{t <n}\bbE_h \bigl( \sfX ; \sfT= t\bigr) \bbP_h^\otimes\Bigl( \sum_1^{N-1} \sfT_i = n -t \Bigr) .
\]
By~\eqref{eq:stretched}, ${|\bbE_h \lb \sfX ; \sfT = t\rb|} \leqs t{\rm e}^{-\kappa_Tt^{1/3}}$.  On the other
hand, we have
\[
 \max_{t >0}\sum_{N\leq n} \frac{N}{n}\bbP^\otimes_h\Bigl( \sum_1^{N-1}\sfT_i =n-t\Bigr) \leq \max_{t >0}\bbP^\otimes_h\Bigl(\exists N\geq 0\,:\, \sum_1^{N}\sfT_i = n-t \Bigr) \leq 1 ,
\]
uniformly in $n$. Moreover, for every $t$ fixed, as $n\to\infty$,
\be  
\label{eq:Novern}
 \sum_N  \frac{N}{n}  \bbP^\otimes_h\Bigl( \sum_1^{N-1}\sfT_i =n-t\Bigr) = 
\frac{\lb 1- \frac{t}{n}\rb}{\lb \bbE_h \sfT\rb^2} \lb 1+\smo{1}\rb .
\ee
Indeed, the typical number of steps $N$ required to produce the total $\sfT$-length $(n-t)$ is $(n-t)/\bbE_h\sfT$.
By a stretched LD upper bound (see e.g. Lemma~2.1 in~\cite{Denisov}),
\be  
\label{eq:LDStretched}
\bbP_h^\otimes \Bigl( \bigl|\sum_1^N \sfT_i - N\bbE_h\sfT\bigr|> m \Bigr) \leq 
N{\rm exp}\lbr -c_{13} \frac{m}{\sqrt{N}}-\kappa_T m^{1/3}\rbr . 
\ee
Therefore, one may ignore terms on the left-hand side of~\eqref{eq:Novern} with
\[
\abs{N- (n-t)/\bbE_h \sfT} >N^{1/2 +\delta} .
\]
 For the remaining terms, $N/n = (1-t/n )/\bbE_h\sfT (1 +
\smo{1})$, and~\eqref{eq:Novern} follows by the usual renewal theory.
Since $\sum_t \bbE_h \bigl( \sfX ; \sfT= t\bigr)  = \bbE_h \sfX$, ~\eqref{eq:EfVelocity} follows.

\subsection{Law of large number}
We turn to the proof of~\eqref{eq:LLN}. By~\eqref{eq:Xdecay},
\[
 \log \bbP_h^\otimes\Bigl( \bigl|\sum_1^N \sfX_i - 
N\bbE_h\sfX\bigr|>N^{1/2 +\delta^\prime}\Bigr) \leqs -N^{2\delta^\prime } .
\]
By~\eqref{eq:LDStretched}
\[
 \log \bbP_h^\otimes\Bigl( \bigl|\sum_1^N \sfT_i - N\bbE_h\sfT\bigr|>N^{1/2 +\delta}\Bigr) \leqs 
-N^{\min\lbr \delta , \frac{1+2\delta}{6}\rbr }.
\]
It remains to take $0 <\delta <\delta^\prime $, and~\eqref{eq:LLN} follows.

\subsection{Central Limit Theorem}
\label{sub:CLT}
Let us now turn to the proof of~\eqref{eq:CLT}.
Let $\theta\in\bbR^d$. By~\eqref{eq:remBoundary}, \eqref{eq:bondary-weights} and~\eqref{eq:criticalAhn},
the characteristic function
\begin{multline*}
\frac{1}{A_h^n}\sum_xA_h^n (x){\rm e}^{\frac{i\theta\cdot (x-nv)}{\sqrt{n}}}\\
 = \bbE_h \sfT \sum_x\sum_N\bbP^\otimes_h 
\Bigl( \sum_1^N\sfX_i = x ;\sum_1^N\sfT_i = n\Bigr)
{\rm e}^{\frac{i\theta\cdot (x-nv)}{\sqrt{n}}}\lb 1+\smo{1}\rb .
\end{multline*}
Consider the $(d+1)$-dimensional renewal relation:
\[
 \frt (x,n) \df \sum_N\bbP^\otimes_h 
\Bigl( \sum_1^N\sfX_i = x ;\sum_1^N\sfT_i = n\Bigr) = \sum_{y,m}\frt (x-y ,n-m )\bbP_h \lb \sfX = y; \sfT=m\rb .
\]
Under $\bbP_h$ the vector $\lb\sfX ,\sfT \rb$ has a proper $(d+1)$-dimensional distribution with
exponential tails in $\sfX$ and stretched exponential tails in $\sfT$. 
Hence classical multi-dimensional
renewal theory (e.g. \cite{Stam}) applies. Namely, in our notation, 
Theorem~3.2 in~\cite{Stam} implies: Let $\Sigma$ be the $d$-dimensional covariance matrix
 (under $\bbP_h$) of  $\lb \sfX - v (h )\sfT\rb/\sqrt{\bbE_h \sfT}$. 
Set
\[
 W_h (x ,n ) = \frac{1}{\sqrt{(2\pi n)^d{\rm det}\lb \Sigma\rb} } 
{\rm exp}\Bigl\{ -\frac{1}{2n}\Sigma^{-1} (x- n v(h))\cdot (x-n v (h))\Bigr\}
\]
Then, 
\begin{equation}
 \label{eq:Stam}
\lim_{n\to\infty}
\Bigl(  \frt (x, n) - \frac{1}{\bbE_h\sfT}
W_h (x,n) 
\Bigr)  = 0
\end{equation}
uniformly in $x\in\bbZ^d$. ~\eqref{eq:CLT} follows.

\subsection{First order transition at critical drifts}
\label{sub:first}
Let $h\in\bar\Ka^c$.
For any $g\in\bar\Ka^c$ the trajectories $\gamma$ under $\bbA_g^n$  have the very same irreducible
structure as in~\eqref{eq:renewal}.   In fact, for all 
such $g$ which in addition  are sufficiently close to $h$, 
it is possible to use the very same
forward cone $\calY_h$ and, consequently, the very same set 
of irreducible paths $\calT$.  Therefore, our results here
(for critical drifts) and the analysis of super-critical drifts~\cite{IoffeVelenik-Annealed}
imply  that 
for each $g\in\bar\Ka^c$ there exists a number $\lambda = \lambda (g) \in [0,\infty )$
such that for all $g$ sufficiently close to $h$, the weights 
\[
 a_{g} (\gamma ) \df a_h (\gamma ){\rm e}^{(g- h)\cdot\sfX (\gamma ) - (\lambda (g ) 
-\lambda (h) ) \sfT (\gamma )}
\]
give rise to a probability distribution $\bbP_{g}$ on the set $\calT$ of irreducible trajectories. 
The function $g\to\lambda (g )$ is   convex  \cite{Zerner,IoffeVelenik-Annealed} and hence continuous. 
Furthermore $\lambda (\cdot )\equiv 0$ on $\partial \Ka$. 
As in~\eqref{eq:l-velocity}, $v(g ) = \bbE_{g }\sfX/\bbE_{g }\sfT$.
In view of  exponential decay of the tails of $\sfX$ under $\bbP_h(.)$ for $h\not\in \Ka$ 
\cite{IoffeVelenik-Annealed}, and in view 
of the tails estimates \eqref{eq:Xdecay} and \eqref{eq:stretched} on $\sfX$ and $\sfT$ under 
$\bbP_h(.)$ in the case of critical drifts $h\in\partial\Ka$, we infer that
\[
 \lim_{g\to h}\bbE_{g} \lb \sfX ,\sfT\rb = \bbE_{h} \lb \sfX ,\sfT\rb ,
\]
and the claim follows. 

\subsection{Crossing random walks} 
The proof of Theorem~\ref{thm:C} goes along the lines of the proof of Theorem~\ref{thm:A}. It is
actually even slightly simpler since we do not need to control ``short'' walks as in~\eqref{eq:two-sums}
and may directly work with the irreducible decomposition~\eqref{eq:renewal}.

\appendix
\section{Some estimates for the simple random walk}
\label{app:RW}
We collect in this appendix some standard SRW estimates that are used in the main text. Remember that we denote by $G_\Lambda(0,x)$ the Green function of the SRW starting at $0$ and killed as it exits the box $\Lambda\subset\bbZ^d$, i.e.,
\[
G_\Lambda(0,x) = E_{\rm SRW}^0 \bigl[ \sum_{n=0}^{\tau_\Lambda-1} \1_{\{S_n=x\}} \bigr],
\]
with $E_{\rm SRW}^y$ the distribution of the SRW $(S_n)_{n\geq 0}$ starting at $y$, and $\tau_\Lambda = \inf\{n\geq 0\,:\, S_n\not\in\Lambda\}$. We shall use the following estimates.
\begin{proposition}
Assume that $0,x\in\Lambda\subset\Zd$. Then, uniformly in such $x$ and $\Lambda$,
\begin{enumerate}
 \item
\begin{equation}
G_\Lambda(0,x) \leq G_\Lambda(x,x) \leqs
\begin{cases}
\log d(x,\Zd\setminus\Lambda)	&	(d=2),\\
1				&	(d\geq 3).
\end{cases}
\label{eq:GFLambda}
\end{equation}
\item Let $\Lambda\subset\Zd$ containing $0$ be such that $d(x,\Zd\setminus\Lambda)\leq L$ for all $x\in\Lambda$. Then, in any dimension,
\begin{equation}
\log P_{\rm SRW}^x(\tau_\Lambda>n) \leqs n/L^2,
\label{eq:ProbStayLongInSmallSet}
\end{equation}
uniformly in $x\in\Lambda$.
\item Let $\calC$ be a bounded convex subset of $\Rd$, centrally symmetric w.r.t. $0$ and $\tau_C=\inf\setof{n\geq 1}{S_s\not\in\calC}$. Then,
\begin{equation}
\inf_{y\in\Lpartial\calC} P_{\rm SRW}^0 (S_{\tau_C}=y) \simeq \sup_{y\in\Lpartial\calC} P_{\rm SRW}^0 (S_{\tau_C}=y).
\label{eq:EssentiallyUniformExitProb}
\end{equation}
\item Let $\calC$ be a bounded subset of $\Rd$, containing $0$ in its interior, and $\tau_{N}=\inf\setof{n\geq 1}{S_s\not\in N\calC}$. Then, for $N$ large enough,
\begin{equation}
E_{\rm SRW}^0 \#\bigl(\setof{S_k}{0\leq k\leq \tau_{2N}} \cap N\calC \bigr) \geqs
\begin{cases}
N^2/\log N	&	(d=2),\\
N^2		&	(d\geq 3).
\end{cases}
\label{eq:LowerBoundRange}
\end{equation}
\end{enumerate}
\end{proposition}
\begin{proof}
First of all, the inequality $G_\Lambda(0,x) \leq G_\Lambda(x,x)$ is a direct consequence of the strong Markov property. We shall use the following well-known bounds: 
\begin{gather*}
G_{\Zd\setminus\{x\}}(0,0) \leqs \log\abs{x}\qquad(d=2),\\
G_{\Zd}(0,0) \leqs 1\qquad(d\geq 3).
\end{gather*}
The first bound can be found in~\cite[Theorem~4.4.4 and (4.31)]{LawlerLimic}, while the second follows directly from transience. The first claim follows immediately from these estimates and the obvious monotonicity property: $G_{\Lambda_1}(x,x)\leq G_{\Lambda_2}(x,x)$ when $\Lambda_1\subset\Lambda_2$.

The second claim is a consequence of the strong Markov property and the fact that, by the CLT,
\[
\inf_{y\in\Lambda} P_{\rm SRW}^y(\tau_\Lambda\leq L^2) > 0,
\]
uniformly in $L$.

The third claim for square boxes  is Lemma~1.7.4 in \cite{Lawler}. The proof there is adjustable to
the case of centrally symmetric bounded convex domains. 

Choosing $\epsilon$ small enough (independently of $N$) so that $P_{\rm SRW}^0(\tau_N>\epsilon N^2)\geq 1/2$, the last claim follows from
\[
E_{\rm SRW}^0 \#\bigl(\setof{S_k}{0\leq k\leq \tau_{2N}} \cap N\calC \bigr)
\geq
\tfrac12 E_{\rm SRW}^0 \#\setof{S_k}{0\leq k\leq \epsilon N^2}
\]
and the bounds~\cite{DvoretzkyErdos}
\[
E_{\rm SRW}^0 \#\setof{S_k}{0\leq k\leq \epsilon N^2} \geqs
\begin{cases}
N^2/\log N	&	(d=2),\\
N^2		&	(d\geq 3).
\end{cases}
\]
\qed
\end{proof}


\medskip 


\begin{thebibliography}{10}
\bibliographystyle{plain}

\bibitem{Denisov}
Denis Denisov, A.B.  Dieker and Vsevolod Shneer. 
\newblock Large deviations for random walks under subexponentiality: The big-jump domain.
\newblock  \emph{Ann. Probab.}, 36(5):1946--1991, 2008.

\bibitem{DvoretzkyErdos}
Aryeh Dvoretzky and Paul Erd\H{o}s.
\newblock Some problems on random walk in space.
\newblock \emph{Proceedings of the {S}econd {B}erkeley {S}ymposium on
{M}athematical {S}tatistics and {P}robability, 1950}, 353--367. University of California Press, Berkeley and Los Angeles, 1951. 

\bibitem{Flury-LD}
Markus Flury.
\newblock Large deviations and phase transition for random walks in random
  nonnegative potentials.
\newblock \emph{Stochastic Process. Appl.}, 117(5):596--612, 2007.

\bibitem{GrassbergerHsu}
Peter Grassberger and Hsiao-Ping Hsu.
\newblock Stretched Polymers in a Poor Solvent.
\newblock \emph{Physical Review E}, 65(3):031807, 2002.

\bibitem{IoffeVelenik-Annealed}
Dmitry Ioffe and Yvan Velenik.
\newblock Ballistic phase of self-interacting random walks.
\newblock In \emph{Analysis and stochastics of growth processes and interface
  models}, pages 55--79. Oxford Univ. Press, Oxford, 2008.

\bibitem{IV-BPS}
Dmitry Ioffe and Yvan Velenik. 
\newblock The statistical mechanics of stretched polymers.
\newblock  \emph{Braz. J. Probab. Stat.}, 24(2):279--299, 2010.

\bibitem{Kaup}
Ludger Kaup and Burchard Kaup.
\newblock \emph{Holomorphic functions of several variables}, volume~3 of \emph{de
  Gruyter Studies in Mathematics}.
\newblock Walter de Gruyter \& Co., Berlin, 1983.

\bibitem{Kesten}
Harry Kesten.
\newblock \emph{Percolation theory for mathematicians}, 
\newblock Birkh\"{a}user, Berlin, 1982.



\bibitem{KosyginaMountford}
Elena Kosygina and Thomas Mountford.
\newblock Crossing velocities for an annealed random walk in a random potential.
\newblock To appear in Stochastic Processes and their Applications;
\newblock arXiv:1103.0515, 2011. 

\bibitem{LawlerLimic}
Gregory F. Lawler and  Vlada Limic.
\newblock Random walk: a modern introduction.
\newblock Cambridge Studies in Advanced Mathematics, 123. Cambridge University Press, Cambridge,  2010.

\bibitem{Lawler}
Gregory F. Lawler. 
\newblock \emph{Intersections of random walks},
\newblock Probability and its Applications,
\newblock Birkh\"auser, Boston, MA, 1991. 
      

\bibitem{MehraGrassberger}
Vishal Mehra and Peter Grassberger.
\newblock Transition to localization of biased walkers in a randomly absorbing environment.
\newblock \emph{Physica D}, 168-169:244--257, 2002.

\bibitem{Stam}
Aaart J.~Stam, 
\newblock Renewal theory in $r$ dimensions. II. 
\newblock \emph{ Compositio Math.}, 23:1--13, 1971.

\bibitem{Sznitman-book}
Alain-Sol Sznitman.
\newblock \emph{Brownian motion, obstacles and random media}.
\newblock Springer Monographs in Mathematics. Springer-Verlag, Berlin, 1998.

\bibitem{Vermet}
Franck Vermet. 
\newblock Phase transition and law of large numbers for a non-symmetric
              one-dimensional random walk with self-interactions.
\newblock \emph{J. Appl. Probab.}, 35(1):55--63, 1998.

\bibitem{Zerner}
Martin P.~W. Zerner.
\newblock Directional decay of the {G}reen's function for a random nonnegative
  potential on {${\bf Z}^d$}.
\newblock \emph{Ann. Appl. Probab.}, 8(1):246--280, 1998.

\end{thebibliography}
\end{document}